\documentclass[a4paper,11pt]{article}
\usepackage[hidelinks]{hyperref}
\usepackage[T1]{fontenc}
\usepackage[utf8]{inputenc}
\usepackage{graphicx}
\usepackage[margin=1in]{geometry}
\usepackage{amsmath,amsthm}
\usepackage{amssymb, amsfonts}
\usepackage{mathrsfs}
\usepackage{url}
\usepackage{xcolor}
\usepackage{latexsym}
\usepackage[ruled,vlined]{algorithm2e}
\usepackage{color}
\usepackage{subfigure}
\usepackage[section]{placeins}
\usepackage{natbib}
\usepackage{multirow}
\usepackage{float}
\usepackage{algpseudocode}

\usepackage{caption}
\usepackage{verbatim}
\usepackage{threeparttable}
\definecolor{darkblue}{rgb}{0,0,0.4}

\newtheorem{theorem}{Theorem}[section]

\graphicspath{{./}} 
\def\bx{{\mathbf x}}
\def\by{{\mathbf y}}
\begin{document}
\title{A Thresholding Based Operator-Splitting Method for Curvature-Regularized Surface Reconstruction}

\author{Yang Hu\thanks{Department of Mathematics, Hong Kong Baptist University, Kowloon Tong, Hong Kong.
		Email: 24481408@life.hkbu.edu.hk.}, 
Hao Liu\thanks{Department of Mathematics, Hong Kong Baptist University, Kowloon Tong, Hong Kong.
		Email: haoliu@hkbu.edu.hk.}, 
Dong Wang\thanks{School of Science and Engineering, The Chinese University of Hong Kong (Shenzhen), Shenzhen, Guangdong 518172, China \& Shenzhen International Center for Industrial and Applied Mathematics, Shenzhen Research Institute of Big Data, Guangdong 518172, China \& Shenzhen Loop Area Institute, Guangdong 518048, China. Email: wangdong@cuhk.edu.cn.}, 
Tieyong Zeng\thanks{Institute for Advanced Study, Beijing Normal-Hong Kong Baptist University, Zhuhai, Guangdong, China. School of Mathematics and Statistics, Guangzhou Nanfang College, Guangzhou, Guangdong, China. Email: tieyongzeng@bnbu.edu.cn.}}

\date{}
\maketitle

\begin{abstract}
    Surface reconstruction from point clouds is a fundamental problem in computational geometry with broad applications in computer graphics, medical imaging, and manufacturing. In this paper, we propose an efficient method for solving a curvature-regularized surface reconstruction model. To avoid the computational cost associated with reinitialization in traditional level set methods, we represent the reconstructed surface by an indicator function. By introducing an auxiliary variable, we reformulate the original optimization problem as the computation of the steady-state solution of an initial value problem. We then develop an operator-splitting method to decompose the resulting problem into two tractable subproblems, one of which can be efficiently solved by iterative thresholding. The proposed approach combines the advantages of curvature regularization, operator splitting, and threshold dynamics. Numerical experiments show that the proposed method is computationally efficient and can accurately reconstruct sharp corners and concave features.
\end{abstract}
\section{Introduction}
Over the past few decades, surface reconstruction from point clouds has attracted considerable attention in computational geometry and related fields \cite{bole1991three}. Point clouds are typically acquired through optical sensing techniques such as laser scanning. Recovering an accurate surface from such data is a fundamental task in a wide range of applications, including computer graphics \cite{wang1991characterizing, calakli2011ssd}, medical imaging \cite{khan2018single}, and manufacturing \cite{bi2010advances}, and many others \cite{berger2016survey}. The main objective of surface reconstruction is to reconstruct a meaningful and accurate surface that faithfully preserves the geometric features of the underlying point cloud, while maintaining high computational efficiency.

A seminal contribution in this area is \cite{zhao2000implicit}, which introduced the classical distance-based implicit surface model:
\begin{align}
    E(\Gamma) = \left(\int_{\Gamma} |d(\mathbf{x})|^{p}\, ds \right)^{1/p},
    \label{eq.model.dis}
\end{align}
where \(\Gamma\) denotes the surface to be reconstructed, \(d(\mathbf{x})\) is the distance from a point \(\mathbf{x}\) to the point cloud, and \(ds\) is the surface element. The reconstructed surface is obtained by seeking a surface \(\Gamma\) that minimizes \(E(\Gamma)\). In \cite{zhao2000implicit}, the level set method was adopted to represent the surface and solve the resulting optimization problem. Motivated by \eqref{eq.model.dis}, \cite{he2020curvature} introduced a curvature-regularized model that better captures concave features and sharp corners, while \cite{law2025approximated} proposed an \(L_1\)-curvature-regularized model. To reconstruct surfaces from incomplete and noisy data, \cite{liu2026pca} developed a PCA-based variational model. Another PCA-based model for dimension reduction on Riemannian manifolds was proposed in \cite{liu2017level}. In \cite{liang2013robust}, the authors formulated a model inspired by image segmentation, in which the distance function is used to construct an edge-indicator function. A directional \(G\)-norm-based model was proposed in \cite{cui2023surface} for reconstructing surfaces from noisy point clouds. When normal information is available, \cite{liu2008implicit} introduced a regularization term that aligns the surface normal with the given normals, and \cite{kazhdan2006poisson} proposed Poisson surface reconstruction by solving a Poisson equation. Other approaches include graph-cut-based methods \cite{hornung2006robust, wan2012reconstructing,shi2012curvature,paris2006surface} and methods based on the Gauss formula \cite{lin2022surface,lu2018surface}. Related surface fairing problems have also been studied in \cite{lai2013ridge,brito2013fast,LiuTaiGlowinski2022}.

Surface processing problems usually involve nonconvex optimization and are therefore difficult to solve directly. To address this challenge, several efficient algorithms have been developed. One important class is the alternating direction method of multipliers (ADMM), which introduces Lagrange multipliers and updates the variables alternately. ADMM has been used in \cite{law2025approximated,cui2023surface,lai2013ridge,estellers2012efficient}. Another important class is operator-splitting methods \cite{he2020curvature,liu2026pca}, which decompose a complicated problem into several easier subproblems. It has been shown in \cite{deng2019new,duan2022fast} that operator-splitting methods can be more efficient than ADMM. Recently, an accelerated operator-splitting method was proposed in \cite{wu2024variational}.

Many of the aforementioned methods use the level set method to represent surfaces. For numerical stability, the level set function is typically required to remain close to a signed distance function, and a standard way to enforce this property is reinitialization \cite{osher2003levelset}. However, reinitialization requires solving a local first-order initial value problem every few iterations, which incurs a considerable computational cost. To address this issue, \cite{wang2021iterative} employed a phase-field representation of the surface and used threshold dynamics to solve \eqref{eq.model.dis}. Threshold dynamics has been widely used in problems such as interface motion, wetting, and image processing \cite{merriman1994motion,wang2019improved,esedoglu2006threshold,wang2017iterative}. In \cite{wang2021iterative}, surfaces are represented by indicator functions, and the target energy is approximated via short-time heat flow. Consequently, the resulting algorithm replaces reinitialization with iterative thresholding, leading to a significant improvement in computational efficiency. Nevertheless, this approach is restricted to the basic distance-based model \eqref{eq.model.dis} and cannot be directly extended to models with nonlinear regularization terms such as curvature.

Motivated by \cite{wang2021iterative,he2020curvature}, we propose an efficient method for a curvature-regularized surface reconstruction model. To avoid the additional computational cost associated with reinitialization, we adopt the phase-field formulation and represent the surface by an indicator function. We first introduce an auxiliary variable to decouple the nonlinearity and reformulate the resulting optimization problem as the computation of the steady-state solution of an initial value problem. We then apply an operator-splitting method for time discretization, which leads to two subproblems that can both be solved efficiently. The proposed approach can be viewed as a hybrid of operator-splitting method and threshold dynamics, since one of the subproblems is solved by iterative thresholding. Our main contributions are summarized as follows:
\begin{enumerate}
    \item We propose a curvature-regularized surface reconstruction method based on a phase-field formulation, operator splitting, and thresholding. The use of an indicator-function representation completely eliminates the reinitialization step required by traditional level set methods, while preserving the ability of curvature regularization to recover sharp corners, edges, and concave surface features.

    \item Through extensive numerical experiments, we demonstrate that the proposed method faithfully reconstructs surfaces with complex geometric features from point clouds, achieving competitive reconstruction quality with substantially reduced computational cost compared to level-set-based approaches.

    \item The proposed framework establishes a new connection between threshold dynamics and curvature-regularized variational models, providing a principled pathway for extending thresholding-based algorithms to a broader class of surface reconstruction 
    problems.
\end{enumerate}

The remainder of this paper is organized as follows. Section \ref{sec.model} introduces the curvature-regularized model and its phase-field formulation, together with auxiliary variables for decoupling the nonlinearity. Section \ref{sec.scheme} presents the proposed operator-splitting method for the reformulated problem. In particular, Section \ref{sec.optcondition} reformulates the problem as an initial value problem, and Section \ref{sec.osm} develops a two-step splitting strategy, leading to two subproblems for each iteration. The solution to each subproblem is discussed in Section \ref{sec.subproblem}. Numerical implementation details are presented in Section \ref{sec.Numerical}. Section \ref{sec.experiment} reports systematic two-dimensional and three-dimensional experiments to demonstrate the effectiveness of the proposed method and compare it with existing approaches. Finally, Section \ref{sec.conclusion} concludes the paper.

\section{A phase-field formulation of the curvature based model}
\label{sec.model}
Let $\mathcal{C} \subset \mathbb{R}^n$ (for $n=2$ or $3$) be the given set of unorganized data points. We define a distance function $d: \Omega \to \mathbb{R}$ on the computational domain $\Omega \subset \mathbb{R}^n$ as:
\begin{align}
d(\mathbf{x}) = \inf_{\mathbf{y} \in \mathcal{C}} |\mathbf{x} - \mathbf{y}|,
\label{eq:distance}
\end{align}
which measures the distance from any point $\mathbf{x}$ to the point cloud.
We seek a surface $\Gamma$ (a closed curve in two-dimensional space or a surface in three-dimensional space) that represents the point cloud. This is formulated as a variational problem where the optimal surface minimizes an energy functional comprising a data-fidelity term and a curvature-based regularization term \cite{he2020curvature}:
\begin{align}
E(\Gamma) = \int_{\Gamma} d^{\alpha}(\mathbf{x}) ds + \eta \int_{\Gamma} \kappa^{\beta}(\Gamma)  ds,
\label{eq.ener.0}
\end{align}
where $ds$ denotes the surface area element. The first term $\int_{\Gamma} d^{\alpha}(\mathbf{x}) ds$ drives the surface toward the point cloud by penalizing the distance to $\mathcal{C}$, with the exponent $\alpha \ge 0$ controlling the sensitivity of the penalty. The second term, $\eta \int_{\Gamma} \kappa^{\beta}(\Gamma)  ds$, regularizes the surface by penalizing high curvature, where $\kappa(\Gamma)$ is the mean curvature of $\Gamma$. The parameter $\eta \ge 0$ balances the influence of the regularization, and the exponent $\beta \ge 0$ modulates the penalty on curvature. We choose even values of $\beta$ so that the curvature term is positive.

To deal with topological changes and represent the interface implicitly, we employ an indicator function $u: \Omega \to \{0,1\}$:
\begin{align}
u(\mathbf{x}) = \begin{cases}
1 & \text{if }  \mathbf{x} \in \Omega_{\Gamma},\\
0 & \text{otherwise},
\end{cases}
\label{eq:indicator}
\end{align}
where $\Omega_{\Gamma}$ is the region enclosed by $\Gamma$. In this representation, the surface $\Gamma$ is implicitly defined as the discontinuity set of $u$.

It is difficult to directly evaluate and optimize $E(\Gamma)$ since it is a surface integral. We resolve this difficulty by connecting $E(\Gamma)$ to $u$. Specifically, we use a result from \cite{Miranda2007} that relates the integral of a function over a surface to a short-time heat flow. For a small parameter $\tau > 0$, the boundary integral of a quantity $f$ can be approximated by $\sqrt{\pi/\tau} \int_{\Omega} f  u[G_{\tau} * (1-u)]  d\mathbf{x}$, where $G_{\tau}$ is the Gaussian kernel:
\begin{align}
G_{\tau}(\mathbf{x}) = \frac{1}{(4\pi\tau)^{n/2}} \exp\left(-\frac{|\mathbf{x}|^2}{4\tau}\right).
\label{eq:heat_kernel}
\end{align}
Following \cite{Hu2020ThresholdDynamics}, we approximate the energy functional \eqref{eq.ener.0} by the following diffuse energy functional:
\begin{align}
E^{\tau}(u):=\sqrt{\frac{\pi}{\tau}}\int_{\Omega}\sqrt{d^{\alpha}+\eta \kappa^{\beta}(u)}\, u
\left[G_{\tau}\ast \left(\sqrt{d^{\alpha}+\eta \kappa^{\beta}(u)}(1-u)\right) \right] d\mathbf{x}.
\label{eq.ener.10}
\end{align}
Here, \(\kappa(u)\) denotes a numerical approximation of the mean curvature computed from \(u\). It was shown in \cite[Theorem 4]{Hu2020ThresholdDynamics} that, as \(\tau \to 0\), \(E^{\tau}(u)\) converges to the original energy functional in  \eqref{eq.ener.0}.

The energy $E^{\tau}(u)$ in \eqref{eq.ener.10} is highly nonlinear and nonconvex in $u$. To manage this complexity, we introduce an auxiliary variable $q$ to represent the curvature term, thereby decoupling the nonlinearity. This leads to a constrained optimization energy:
\begin{align}
    \min_{u \in \mathcal{B},  q \in \mathcal{K},q = \kappa(u)} \sqrt{\frac{\pi}{\tau}}\int_{\Omega}\sqrt{d^{\alpha}+\eta q^{\beta}}  u  \left[G_{\tau}\ast \left( \sqrt{d^{\alpha}+\eta q^{\beta}} (1-u) \right) \right] d\mathbf{x},
\label{eq.ener.2}
\end{align}
where 
$$
\mathcal{B} := {u \in BV(\Omega, \{0,1\}) }, \quad \mathcal{K} := {q \in BV(\Omega, \mathbb{R}) }
$$
with $BV(\Omega, \mathbb{R})$ being the space of functions of bounded variation.

We remove the constraint by introducing characteristic functions.
Define the sets and their corresponding characteristic functions as
\begin{align*}
	&\Sigma: = \{u \in BV(\Omega,\mathbb{R})|u \in \{0,1\}\}, \quad I_{\Sigma}(u)= \begin{cases}
		0 & \mbox{ if } u\in \Sigma,\\
		+\infty & \mbox{ if } u\notin \Sigma,
	\end{cases}
    \\&S:=\{(u,q)| u \in \Sigma, q \in \mathcal{K}, q = \kappa(u)\}, \quad I_S(u,q)=\begin{cases}
		0 & \mbox{ if } (u,q)\in S,\\
		+\infty & \mbox{ if } (u,q)\notin S.
    \end{cases}
\end{align*}
Denote 
\begin{align}
	J(u,q)=\sqrt{\frac{\pi}{\tau}}\int_{\Omega}\sqrt{d^{\alpha}+\eta q^{\beta}}u\left[G_{\tau}\ast \left(\sqrt{d^{\alpha}+\eta q^{\beta}}(1-u) \right) \right]d\mathbf{x}.
\end{align}

Problem (\ref{eq.ener.2}) is equivalent to the unconstrained problem:
\begin{align}
	\min\limits_{u ,q} J(u,q)+ I_S(u,q)+I_{\Sigma}(u).
	\label{eq.ener.unconstrained}
\end{align}

\section{An operator-splitting method for problem (\ref{eq.ener.unconstrained})}
\label{sec.scheme}
We propose an operator-splitting method to solve problem (\ref{eq.ener.unconstrained}). We first associate problem (\ref{eq.ener.unconstrained}) with an initial value problem. In order to update the numerical solution in time, one needs to solve nonlinear equations. We use an operator-splitting method to decompose the complicated problem into two easy-to-solve subproblems, each of which either admits a closed-form solution or can be solved efficiently.

\subsection{An associated initial-value problem}
\label{sec.optcondition}
If $(u^*,q^*)$ solves (\ref{eq.ener.unconstrained}), it satisfies the optimality condition 
\begin{align}
	\begin{cases}
		D_u J(u^*,q^*)+\partial_u I_S(u^*,q^*)+ \partial_u I_{\Sigma}(u^*)\ni 0,\\
		D_q J(u^*,q^*)+ \partial_q I_S(u^*,q^*) \ni 0,
	\end{cases}
\label{eq.opti}
\end{align}
where $D$ denotes the differentiation and $\partial$ denotes the subdifferential.

We associate (\ref{eq.opti}) with the following initial value problem in the form of a gradient flow:
\begin{align}
	\begin{cases}
		\frac{\partial u}{\partial t}+D_u J(u,q)+\partial_u I_S(u,q)+ \partial_u I_{\Sigma}(u)\ni 0,\\
		\gamma\frac{\partial q}{\partial t}+ D_q J(u,q)+ \partial_q I_S(u,q) \ni 0,\\
		u(0)=u_0,\ q(0)=q_0,
	\end{cases}
\label{eq.ivp}
\end{align}
where $(u_0,q_0)$ is an initial condition and $\gamma$ is a constant controlling the evolving speed of $q$. Then solving (\ref{eq.ener.unconstrained}) is converted to finding the steady state solution of (\ref{eq.ivp}).

\subsection{An operator splitting method for the initial-value problem}
\label{sec.osm}
Problem (\ref{eq.ivp}) is well suited to be solved by operator splitting methods. In this paper, we adopt the Lie scheme. We refer the readers to \cite{glowinski2017some,glowinski2017splitting} for a comprehensive discussion of operator-splitting methods. Let $\Delta t$ be the time step, and $t^n=n\Delta t$. We denote the numerical solution at $t^n$ by $(u^n,q^n)$. Given $(u^n,q^n)$, we update $(u^{n+1},q^{n+1})$ by two substeps:\\
\noindent Substep 1: Solve
\begin{align}
	\begin{cases}
		\begin{cases}
			\frac{\partial u}{\partial t}+D_u J(u,q)+ \partial_u I_{\Sigma}(u)\ni 0\\
			\gamma\frac{\partial q}{\partial t}+ D_q J(u,q) \ni 0\\
		\end{cases}
		\mbox{ in } \Omega\times[t^n,t^{n+1}],\\
		u(t^n)=u^n,\ q(t^n)=q^n,
	\end{cases}
	\label{eq.split.1}
\end{align}
and set 
\begin{align}
	u^{n+1/2}=u(t^{n+1}), \ q^{n+1/2}=q(t^{n+1}).
\end{align}

\noindent Substep 2: Solve
\begin{align}
	\begin{cases}
		\begin{cases}
			\frac{\partial u}{\partial t}+\partial_u I_S(u,q)\ni 0\\
			\gamma\frac{\partial q}{\partial t} +\partial_q I_S(u,q) \ni 0
		\end{cases}
		\mbox{ in } \Omega\times[t^n,t^{n+1}],\\
		u(t^n)=u^{n+1/2},\ q(t^n)=q^{n+1/2},
	\end{cases}
	\label{eq.split.2}
\end{align}
and set 
\begin{align}
	u^{n+1}=u(t^{n+1}), \ q^{n+1}=q(t^{n+1}).
\end{align}
Problem (\ref{eq.split.1}) and (\ref{eq.split.2}) are semi-constructive. One still needs to solve the subproblems. Here we use the Marchuk-Yanenko type scheme to time-discretize (\ref{eq.split.1}) and (\ref{eq.split.2}) by a backward scheme. The resulting scheme reads as
\begin{align}
	&\begin{cases}
		\frac{u^{n+1/2}-u^n}{\Delta t}+D_u J(u^{n+1/2},q^n)+ \partial_q I_{\Sigma}(u^{n+1/2})\ni 0,\\
		\gamma\frac{q^{n+1/2}-q^n}{\Delta t}+ D_q J(u^{n+1/2},q^{n+1/2}) \ni 0,
	\end{cases}
	\label{eq.splittime.1}\\
	&\begin{cases}
		\frac{u^{n+1}-u^{n+1/2}}{\Delta t}+\partial_u I_S(u^{n+1},q^{n+1}) \ni 0,\\
		\gamma\frac{q^{n+1}-q^{n+1/2}}{\Delta t} +\partial_q I_S(u^{n+1},q^{n+1}) \ni 0.
	\end{cases}
	\label{eq.splittime.2}
\end{align}

\subsection{On the solution to problem (\ref{eq.splittime.1})}
\label{sec.subproblem}
\subsubsection{Computing $u^{n+1/2}$}
In (\ref{eq.splittime.1}), $u^{n+1/2}$ solves

\begin{align}
	u^{n+1/2} =\operatorname*{arg\,min}_{v\in \Sigma} &\frac{1}{2} \int_{\Omega} |v-u^n|^2d\bx \nonumber\\
    &+ \Delta t\sqrt{\frac{ \pi}{\tau}}\int_{\Omega}\sqrt{d^{\alpha}+\eta (q^n)^{\beta}}v \left[G_{\tau}\ast \left(\sqrt{d^{\alpha}+\eta (q^n)^{\beta}}(1-v)\right) \right] d\mathbf{x}.
    \label{eq.psi0.1}
\end{align}
Since $v$ is a binary function, we have \( |v - u^n|^2 = v^2 - 2 u^n v + (u^n)^2 = v - 2 u^n v + (u^n)^2 \), and \( v - 2 u^n v \) can be approximated by \( v G_{\tau} \ast (1 - 2 u^n) \) following \cite{Miranda2007}. We approximate problem (\ref{eq.psi0.1}) by
\begin{align}
    \min_{v\in \Sigma} \phi(v)
    \label{eq.v.sigma}
\end{align}
with
\begin{align}
    \phi(v) = &\frac{1}{2} \int_{\Omega} vG_{\tau^{}}\ast(1-2u^n)d\bx \nonumber\\
    &+ \Delta t\sqrt{\frac{ \pi}{\tau}}\int_{\Omega}\sqrt{d^{\alpha}+\eta (q^n)^{\beta}}v\left[G_{\tau}\ast \left(\sqrt{d^{\alpha}+\eta (q^n)^{\beta}}(1-v) \right) \right] d\mathbf{x}
	\label{eq.u1}.
\end{align}
Define 
$$
\widetilde{\Sigma}: = \{u \in BV(\Omega,\mathbb{R})|u \in [0,1]\}, \quad I_{\widetilde{\Sigma}}(u)= \begin{cases}
		0 & \mbox{ if } u\in \widetilde{\Sigma},\\
		+\infty & \mbox{ if } u\notin \widetilde{\Sigma}.
	\end{cases}
$$
We have the following theorem:
\begin{theorem}\label{thm.phi.sigma}
    Problem (\ref{eq.v.sigma}) and the following one
    \begin{align}
        \min_{v\in \widetilde{\Sigma}} \phi(v)
        \label{eq.v.sigmat}
    \end{align}
    have the same solution.
\end{theorem}
\begin{proof}[Proof of Theorem \ref{thm.phi.sigma}]
    By $\Sigma \subset \widetilde{\Sigma}$, we can verify that:
     \begin{align}
		 \min_{v\in \widetilde{\Sigma}}\phi(v) \leq  \min_{v \in \Sigma}\phi(v).
     \end{align}
     Therefore, it suffices to show that: 
     \begin{align}
		\operatorname*{arg\,min}_{v \in \widetilde{\Sigma}} \phi(v) \in \Sigma.      
        \label{thm.phi1.1}
     \end{align}
     Assume the local minimizer of $\phi(v)$ is $v^{\ast} \in  \widetilde{\Sigma}$ and $v^{\ast} \notin \Sigma$, then there exists a set $\omega \subset \Omega$ with nonzero measure and $c > 0$ such that:
     \begin{align}
		v^{\ast}(\bx) \in (c, 1-c), \ \ \ \ \ \forall \bx \in \omega.
        \label{lemma1.2}
     \end{align}
    Let \( v^{s} = v^{\ast} + s \chi_{\omega} \), where \( \chi_{\omega} \) is the indicator function of \( \omega \). We have \( v^{s} \in \widetilde{\Sigma} \) for any \( |s| < c \). The second-order derivative of \( \phi(v^{s}) \) with respect to the variable \( s \) is: 
     \begin{align}
		\frac{\mathrm{\partial}^2 \phi(v^s)}{\mathrm{\partial} s^2} = -2 \Delta t\sqrt{\frac{\pi}{\tau}} \int_{\Omega} \sqrt{d^{\alpha}+\eta (q^{n})^{\beta}} {\chi_{\omega} G_{\tau} \ast \left(\chi_{\omega}\sqrt{d^{\alpha}+\eta (q^{n})^{\beta}}\right)} \, \mathrm{d}x.
     \end{align}
      Since $\sqrt{d^{\alpha}+\eta (q^{n})^{\beta}} \geq 0$, $\Delta t > 0$ and the heat kernel $G_{\tau}\geq 0$, it follows that 
      \[
      \frac{\mathrm{\partial}^2 \phi(v^s)}{\mathrm{\partial}s^2} < 0\] for all $|s| < c$, and in particular at $s=0$ where $v^s = v^*$. This contradicts the second-order necessary condition for $v^{\ast}$ to be a local minimizer. Therefore, any local minimizer of $\phi(v)$ over $\widetilde{\Sigma}$ must lie in $\Sigma$. Hence,
\begin{align}
\operatorname*{arg\,min}_{v \in \widetilde{\Sigma}} \phi(v) \subset \Sigma.
\end{align}
On the other hand, since $\Sigma \subset \widetilde{\Sigma}$, we immediately obtain
\begin{align}
\operatorname*{arg\,min}_{v \in \widetilde{\Sigma}} \phi(v)
=
\operatorname*{arg\,min}_{v \in \Sigma} \phi(v).
\end{align}
Thus, the relaxed problem and the original binary-valued problem have the same minimizers. 
\end{proof}

By Theorem \ref{thm.phi.sigma}, it  is sufficient to consider the relaxed problem (\ref{eq.v.sigmat}) to solve (\ref{eq.v.sigma}).
We use an iterative scheme to repeatedly solve the linearization of (\ref{eq.v.sigmat}). The following theorem gives the variation of $\phi(v)$ with respect to $v$:

\begin{theorem}
The variation of $\phi(v)$ with respect to $v$ is:   
\begin{align}
\frac{\partial\phi(v)}{\partial v}= G_{\tau} \ast \left(\frac{1}{2} - u^n\right) + \Delta t \sqrt{\frac{\pi}{\tau}} \sqrt{d^{\alpha} + \eta (q^n)^{\beta}} \left[ G_{\tau} \ast \left( \sqrt{d^{\alpha} + \eta (q^n)^{\beta}} (1 - 2 v) \right)\right].
\label{eq.psi.1}
\end{align}
\label{Theorem 4.1}
\end{theorem}
\begin{proof}[Proof of Theorem \ref{Theorem 4.1}]
Denote $P = \sqrt{d^{\alpha}+\eta (q^n)^{\beta}}$. The functional in (\ref{eq.u1}) can be written as:
\begin{align}
      \frac{1}{2} \int_{\Omega} vG_{\tau^{}}\ast(1-2u^n)d\bx + \Delta t \sqrt{\frac{ \pi}{\tau}}\int_{\Omega}PvG_{\tau}\ast [P(1-v)] d\mathbf{x}
     \label{M1}.
\end{align}
We decompose (\ref{M1}) into:
\begin{align*}
    I_1(v) &= \frac{1}{2} \int_{\Omega} vG_{\tau^{}}\ast(1-2u^n)d\bx,
\\
    I_2(v) &= \Delta t \sqrt{\frac{ \pi}{\tau}}\int_{\Omega}PvG_{\tau}\ast [P(1-v) ]d\mathbf{x}.
\end{align*}
The variation of $I_1(v) $ is
\begin{align}
    \frac{\partial I_1(v)}{\partial v} = \ G_{\tau^{}}\ast(\frac{1}{2} - u^n).
    \label{eq.psi.1.0}
\end{align}

Now consider $I_2(v)$. For any perturbation direction $h\in \widetilde{\Sigma}$ and scalar $\epsilon >0$, introducing perturbation $\epsilon h(\bx)$ to $v(\bx)$ gives rise to
\begin{align}
    I_2(v+\epsilon h) = \Delta t \sqrt{\frac{ \pi}{\tau}}\int_{\Omega}P(\bx)[v(\bx)+\epsilon h(\bx)]\int_{\Omega}G_{\tau}(\bx-\by) P(\by)[1-v(\by)-\epsilon h(\by)] d\mathbf{y}d\mathbf{x}.
\end{align}
Denote 
\begin{align*}
    I_{21} &= \Delta t \sqrt{\frac{ \pi}{\tau}}\int_{\Omega}P(\bx)v(\bx)\int_{\Omega}G_{\tau}(\bx-\by) P(\by)[1-v(\by)-\epsilon h(\by)] d\mathbf{y}d\mathbf{x},
\\
    I_{22} &= \Delta t \sqrt{\frac{ \pi}{\tau}}\int_{\Omega}P(\bx)\epsilon h(\bx)\int_{\Omega}G_{\tau}(\bx-\by) P(\by)[1-v(\by)-\epsilon h(\by)] d\mathbf{y}d\mathbf{x}.
\end{align*}
The difference $I_2(v+\epsilon h)-I_2(v)$ can be rewritten as:
\begin{align}
    I_2(v+\epsilon h)-I_2(v) = I_{21}+I_{22}-I_2(v) = I_{21}-I_2(v) +I_{22}
    \label{difference}.
\end{align}
For $I_{21}-I_2(v)$, we have
\begin{align}
    I_{21}-I_2(v) = \Delta t \sqrt{\frac{ \pi}{\tau}}\int_{\Omega}P(\bx)v(\bx)\int_{\Omega}G_{\tau}(\bx-\by) P(\by)(-\epsilon h(\by)) d\mathbf{y}d\mathbf{x}\nonumber
   &\\ = \Delta t \sqrt{\frac{ \pi}{\tau}}\int_{\Omega}\int_{\Omega}P(\by)\epsilon h(\by)G_{\tau}(\by-\bx) P(\bx)(-v(\bx)) d\mathbf{y}d\mathbf{x} \nonumber
   &\\ = \Delta t \sqrt{\frac{ \pi}{\tau}}\int_{\Omega}P(\by)\epsilon h(\by)\int_{\Omega}G_{\tau}(\bx-\by) P(\bx)(-v(\bx)) d\mathbf{y}d\mathbf{x}. 
\label{eq.deravative3.1}
\end{align}
Utilizing the property \(G_{\tau}(\by - \bx) = G_{\tau}(\bx - \by)\) and substituting (\ref{eq.deravative3.1}) into (\ref{difference}) gives rise to
\begin{align}
I_{21}+I_{22}-I_2(v) =& \Delta t \sqrt{\frac{ \pi}{\tau}}\int_{\Omega}P(\bx)\epsilon h(\bx)\int_{\Omega}G_{\tau}(\bx-\by) P(\by)[1-2v(\by)-\varepsilon h(\by)] d\mathbf{y}d\mathbf{x} \nonumber\\
=&D_1 - D_2
\label{eq.deravative3}
\end{align}
with
\begin{align*}
   &D_1=\Delta t \epsilon \sqrt{\frac{ \pi}{\tau}}\int_{\Omega}P(\bx) h(\bx)\int_{\Omega}G_{\tau}(\bx-\by) P(\by)[1-2v(\by)] d\mathbf{y}d\mathbf{x},
   \\
      &D_2 = \Delta t \epsilon^2 \sqrt{\frac{ \pi}{\tau}} \int_{\Omega}P(\bx)h(\bx)\int_{\Omega}G_{\tau}(\bx-\by) P(\by)h(\by) d\mathbf{y}d\mathbf{x}.   
\end{align*}
The functional variation in the direction of $h(\bx)$ is 
\begin{align*}
    \int_{\Omega}\frac{ \partial I_2}{\partial v} h(\bx)dx &=\lim_{\epsilon \xrightarrow{}0} \frac{I_2(v+\epsilon h(\bx))-I_2(v)}{\epsilon} \\
    &= \lim_{\epsilon \xrightarrow{}0}\frac{D_1}{\epsilon} - \lim_{\epsilon \xrightarrow{}0}\frac{D_2}{\epsilon}
    \\& = \Delta t \sqrt{\frac{ \pi}{\tau}}\int_{\Omega}P(\bx)h(\bx)\int_{\Omega}G_{\tau}(\bx-\by) P(\by)[1-2v(\by)] d\by d\bx.
\end{align*}
We thus have
\begin{align}
    \frac{ \partial I_2}{\partial v} = \Delta t \sqrt{\frac{ \pi}{\tau}} PG_{\tau} \ast [ P(1-2v)].
    \label{eq.psi2.3}
\end{align}
Putting (\ref{eq.psi.1.0}) and (\ref{eq.psi2.3}) together finishes the proof.
\end{proof}

Denote $\psi(v)=\frac{\partial\phi(v)}{\partial v}$. By Theorem \ref{Theorem 4.1}, the linearization of functional $\phi$ at v is
\begin{align}
L^{\tau}(w,v) = \int_{\Omega} w \, \psi(v)\, d\mathbf{x}
\label{eq.linearization.2}
\end{align}
Our iterative scheme repeatedly solves (\ref{eq.linearization.2}) to compute $u^{n+1/2}$. Specifically, set $w^0=u^n$. We update:
\begin{align}
w^{k+1}(\mathbf{x}) = \operatorname*{arg\,min}_{w \in \Sigma} L^{\tau}(w,w^k).
\label{eq.updatew}
\end{align}
Denote $\psi^k=\psi(w^k)$. Problem (\ref{eq.updatew}) can be solved pointwisely:
\begin{align*}
w^{k+1}(\mathbf{x}) &= \operatorname*{arg\,min}_{w \in \Sigma} w(\mathbf{x}) \psi^k,
\end{align*}
which can be exactly solved via a thresholding step
\begin{align}
	w^{k+1}(\mathbf{x})=\begin{cases}	
		1 & \mbox{ if } \psi^k\leq 0,\\
		0 & \mbox{ otherwise}.	
	\end{cases}
    \label{algorithm1}
\end{align}
The iteration is terminated when $\max\{w^{k+1} - w^{k}\} = 0$. Denote the converged function by $w^*$, we set $u^{n+1/2}=w^*$. 

Since scheme (\ref{eq.splittime.1})-(\ref{eq.splittime.2}) is an iterative scheme, we do not need to solve (\ref{eq.u1}) exactly. We can only use a few iterations for $w$.

\subsubsection{Computing $q^{n+1/2}$}
\label{section4.2}
In (\ref{eq.splittime.1}), $q^{n+1/2}$ solves
\begin{align}
	q^{n+1/2} = \operatorname*{arg\,min}_{p \in \mathcal{K}} \frac{\gamma}{2} &\int_{\Omega} |p-q^n|^2d\bx \nonumber\\
    &+ \Delta t \sqrt{\frac{ \pi}{\tau}}\int_{\Omega}\sqrt{d^{\alpha}+\eta p^{\beta}}u^{n+1/2}\left[G_{\tau}\ast \left( \sqrt{d^{\alpha}+\eta p^{\beta}}(1-u^{n+1/2}) \right) \right] d\mathbf{x}.
    \label{eq.g1}
\end{align}
The following theorem gives an optimality condition for $q^{n+1/2}$:
\begin{theorem}
For the optimization problem (\ref{eq.g1}), if \( p^{\ast} \) is a minimizer, it satisfies
\begin{align}
	\gamma p^{\ast} - \gamma q^n &+\frac{\beta\eta (p^{\ast})^{\beta-1}}{2\sqrt{d^{\alpha}+\eta (p^{\ast})^{\beta}}}\Delta t\sqrt{\frac{ \pi}{\tau}}u^{n+1/2}\left[G_{\tau}\ast \left(\sqrt{d^{\alpha}+\eta (p^{\ast})^{\beta}}(1-u^{n+1/2})\right)\right]\nonumber\\
    &+\frac{\beta\eta (p^{\ast})^{\beta-1}}{2\sqrt{d^{\alpha}+\eta (p^{\ast})^{\beta}}}\Delta t\sqrt{\frac{ \pi}{\tau}}(1-u^{n+1/2})\left[G_{\tau}\ast \left(\sqrt{d^{\alpha}+\eta (p^{\ast})^{\beta}}u^{n+1/2}\right)\right]  = 0.
    \label{eq.q1.01}
\end{align}
\label{theorem2}
\end{theorem}
\begin{proof}[Proof of Theorem \ref{theorem2}]
    Denote
    \begin{align*}
      S(p) = \sqrt{d^{\alpha}+\eta p^{\beta}}.
    \end{align*}
    We decompose the functional in (\ref{eq.g1}) into:
    \begin{align*}
    &Q_1(p) = \frac{\gamma}{2} \int_{\Omega} |p-q^n|^2\bx,
    \\
    &Q_2(p) = \Delta t \sqrt{\frac{ \pi}{\tau}}\int_{\Omega}S(p)u^{n+1/2}\left[G_{\tau}\ast \left( S(p)(1-u^{n+1/2}) \right) \right] d\mathbf{x}.
    \end{align*}
    It is straightforward that the differential of $Q_1(p)$ is:
    \begin{align}
    D_pQ_1(p) = \gamma p - \gamma q^n.
    \label{eq.Q1grad}
    \end{align}
    By introducing a small perturbation $\epsilon h $ to  $Q_2$, we obtain that
    \begin{align}
    &Q_2(p+\epsilon h) - Q_2(p)\nonumber\\
    =&\Delta t \sqrt{\frac{ \pi}{\tau}}\int_{\Omega}S(p+\epsilon h)u^{n+1/2}\left[G_{\tau}\ast \left( S(p+\epsilon h)(1-u^{n+1/2}) \right) \right] d\mathbf{x} \nonumber
   \\ &-\Delta t \sqrt{\frac{ \pi}{\tau}}\int_{\Omega}S(p)u^{n+1/2}\left[G_{\tau}\ast \left( S(p)(1-u^{n+1/2}) \right) \right] d\mathbf{x}.\nonumber
\\  =& \Delta t \sqrt{\frac{\pi}{\tau}} \int_{\Omega} S(p+\epsilon h) u^{n+1/2} \left[ G_\tau \ast \left( S(p+\epsilon h)(1 - u^{n+1/2}) \right) \right] d\mathbf{x}  \nonumber
\\& - \Delta t \sqrt{\frac{\pi}{\tau}} \int_{\Omega} S(p) u^{n+1/2} \left[ G_\tau \ast \ \left(S(p+\epsilon h)(1 - u^{n+1/2}) \right) \right] d\mathbf{x}
 \nonumber 
\\\nonumber&+  \Delta t \sqrt{\frac{\pi}{\tau}} \int_{\Omega} S(p) u^{n+1/2} \left[ G_\tau \ast \left(  S(p+\epsilon h)(1 - u^{n+1/2}) \right) \right] d\mathbf{x}
\\\nonumber& - \Delta t \sqrt{\frac{\pi}{\tau}} \int_{\Omega} S(p) u^{n+1/2} \left[ G_\tau \ast \left( S(p)(1 - u^{n+1/2}) \right) \right] d\mathbf{x}.
\\ =&\Delta t \sqrt{\frac{\pi}{\tau}} \int_{\Omega} [ S(p+\epsilon h) - S(p) ]u^{n+1/2}  \left[ G_\tau \ast \left( S(p+\epsilon h)(1 - u^{n+1/2}) \right)\right] d\mathbf{x}  \nonumber
 \\& + \Delta t \sqrt{\frac{\pi}{\tau}} \int_{\Omega} S(p) u^{n+1/2} \left[ G_\tau \ast \left( [S(p+\epsilon h) - S(p) ](1 - u^{n+1/2}) \right)\right] d\mathbf{x}.
\label{gradient2}
\end{align}
We thus have
\begin{align}
   \nonumber &\int_{\Omega}D_pQ_2(p) h d\mathbf{x}\\
   =&\lim_{\epsilon \to 0}\frac{Q_2(p+\epsilon h) - Q_2(p)}{\epsilon} \nonumber\\ 
   \nonumber= &\lim_{\epsilon \to 0} \Delta t \sqrt{\frac{\pi}{\tau}}\left\{ \int_{\Omega} \frac{ S(p+\epsilon h) - S(p)}{\epsilon} u^{n+1/2}  \left[ G_\tau \ast \left( S(p+\epsilon h)(1 - u^{n+1/2}) \right)\right] d\mathbf{x}\right\}\\
   &+\lim_{\epsilon \to 0} \Delta t \sqrt{\frac{\pi}{\tau}}\left\{\int_{\Omega} S(p) u^{n+1/2} \left[ G_\tau \ast \left(\frac{ S(p+\epsilon h) - S(p)}{\epsilon}(1 - u^{n+1/2}) \right) \right] d\mathbf{x}
   \right\}.
   \label{gradient3}
\end{align}
Note that the differential of function $S(p)$ is:
\begin{align}
    S'(p) = \lim_{\epsilon \to 0} \frac{ S(p+\epsilon h) - S(p)}{\epsilon} = \frac{\beta\eta p^{\beta-1}}{2\sqrt{d^{\alpha}+\eta p^{\beta}}}.
    \label{eq.Sp}
\end{align}
Combining (\ref{gradient3}) and (\ref{eq.Sp}) gives rise to
\begin{align}
    \nonumber  D_pQ_2(p) =& \Delta t \sqrt{\frac{\pi}{\tau}}   S'(p) u^{n+1/2}  \left[ G_\tau \ast \left( S(p)(1 - u^{n+1/2}) \right)\right] 
   \\&+ \Delta t \sqrt{\frac{\pi}{\tau}} S(p) u^{n+1/2} \left[ G_\tau \ast \left( S'(p)(1 - u^{n+1/2}) \right) \right] 
   .
   \label{gradient4}
\end{align}
Similar to the derivation of (\ref{eq.deravative3}), we can rewrite (\ref{gradient4}) as:
\begin{align}
    \nonumber D_pQ_2(p) =  &\Delta t \sqrt{\frac{\pi}{\tau}}  S'(p) u^{n+1/2}  \left[ G_\tau \ast \left( S(p)(1 - u^{n+1/2}) \right)\right] 
   \\&+ \Delta t \sqrt{\frac{\pi}{\tau}} S'(p)(1 - u^{n+1/2})  \left[ G_\tau \ast \left(S(p) u^{n+1/2} \right) \right] .
   \label{gradient5}
\end{align}
Putting (\ref{eq.Q1grad}) and (\ref{gradient5}) together finishes the proof.
\end{proof}
Based on Theorem \ref{theorem2}, we use a fixed-point method to iteratively solve (\ref{eq.q1.01}). We define $p^0 =  q^n$. At the $k$-th iteration, we compute the $p^{k+1}$ by:
\begin{align}
	 p^{k+1} =  & q^n -\frac{\beta\eta (p^{k})^{\beta-1}}{2\gamma\sqrt{d^{\alpha}+\eta (p^{k})^{\beta}}+c}\Delta t\sqrt{\frac{ \pi}{\tau}}u^{n+1/2}\left[G_{\tau}\ast \left(\sqrt{d^{\alpha}+\eta (p^{k})^{\beta}}(1-u^{n+1/2})\right)\right]\nonumber\\
    &-\frac{\beta\eta (p^{k})^{\beta-1}}{2\gamma\sqrt{d^{\alpha}+\eta (p^{k})^{\beta}}+c}\Delta t\sqrt{\frac{ \pi}{\tau}}(1-u^{n+1/2})\left[G_{\tau}\ast \left(\sqrt{d^{\alpha}+\eta (p^{k})^{\beta}}u^{n+1/2}\right)\right],
    \label{eq.q1.05}
\end{align}
where $c$ is a positive constant used to prevent division by zero. We update $p^k$ until $\|p^{k+1} - p^{k} \| \leq \epsilon$ or \(k \geq k_{\max}\), where $k_{\max}$ is the maximum number of iterations of the fixed-point method. Denote the converged quantity by $p^*$. We set
\begin{align}
   q^{n+1/2} = p^*.
\end{align}

\subsection{On the solution to (\ref{eq.splittime.2})}
\label{section4.3}
In (\ref{eq.splittime.2}), $(u^{n+1},q^{n+1})$ solves
\begin{align*}
	\min_{(v,p)\in S} \frac{1}{2} \int_{\Omega} |v-u^{n+1/2}|^2d\bx + \frac{\gamma}{2} \int_{\Omega} |p-q^{n+1/2}|^2d\bx.
\end{align*}
Since $(u^{n+1},q^{n+1})\in S$, we have $q^{n+1}= \kappa(u^{n+1})$. Equivalently, one can solve
\begin{align}
	\min_{v} \frac{1}{2} \int_{\Omega} |v-u^{n+1/2}|^2d\bx + \frac{\gamma}{2} \int_{\Omega} \left|\kappa(v)-q^{n+1/2}\right|^2d\bx.
	\label{eq.v2.indi}
\end{align}
Note that the surface is represented as the 0.5 level set of $u$. Although $u$ is discontinuous, after discretization, the curvature can still be computed by
\begin{align}
    \kappa(u)=\nabla\cdot \frac{\nabla u}{|\nabla u|}.
\end{align}
Thus, solving (\ref{eq.v2.indi}) requires the computation of fourth-order derivatives of $v$, which is computationally expensive. When $\Delta t$ is small, we expect $(u^{n+1},q^{n+1})$ to stay close to $(u^{n},q^{n})$. We thus approximately solve (\ref{eq.v2.indi}) by
\begin{align}
	\begin{cases}
		u^{n+1}=u^{n+1/2},\\
		q^{n+1}=\kappa(u^{n+1}).
	\end{cases}
\label{eq.v2.simple}
\end{align}
To improve the overall stability, we use a relaxation step for $q^{n+1}$. Specifically, for some $\alpha_0\in [0,1)$, we compute as
\begin{align}
	q^{n+1}=\alpha_0 \kappa(u^{n+1})+ (1-\alpha_0) q^{n+1/2}.
	\label{eq.q2.relax}
\end{align}

\subsection{On the choice of $\tau$}
The Gaussian kernel parameter $\tau$ controls how well the Gaussian convolution in (\ref{eq.ener.10}) controls the boundary integral. Smaller $\tau$ gives higher approximation accuracy, while a larger $\tau$ makes the algorithm more stable at the beginning of the iteration.
In this paper, we follow the strategy used by \cite{wang2021iterative} in which a decreasing sequence of $\tau$ is used. Specifically, we reduce $\tau$ by half when $\|u^{n+1}-u^{n}\|\leq \epsilon$ for some small $\epsilon>0$. Our algorithm is summarized in Algorithm \ref{alg3}.

\begin{algorithm}[t!]
\caption{Total algorithm for updating $(u^n,q^n) \xrightarrow{}(u^{n+1},q^{n+1})$ 
\label{alg3}}
\begin{algorithmic}
\State \textbf{Input}: Initialize $u^0$, $q^0 =\kappa(u^0) $ and $\tau =\tau_0$. $\mathbf{\alpha},\mathbf{\beta}$: Scaling exponent parameter. \textbf{d}: distance space of input point cloud. The stopping criteria coefficient of parameter $\tau$ is $m \geq 1$.  \\
\While{ i $\leq$ $i_{max}$}{
      \State \textbf{Step 1:} Update \((u^{n},q^{n}) \to (u^{n+1/2},q^{n+1/2})\) by (\ref{eq.updatew}) and (\ref{eq.q1.05}).
      \State \textbf{Step 2:} Update \((u^{n+1/2},q^{n+1/2}) \to (u^{n+1},q^{n+1})\) by (\ref{eq.v2.simple}) and (\ref{eq.q2.relax}).
      \State \If{$\|u^{n+1}-u^{n}\|\leq \epsilon$}
    {
        \State$\tau_{new} = \frac{1}{2} \tau$
        \State \If{$\tau \leq \frac{\tau_0}{m}$}{
        \State \textbf{Break}
        }
        \State$\tau = \tau_{new}$
    }
      \State $i = i+1$
}
\label{algorithm3}
\State \textbf{Output}: The converged $u^{n}$.
\end{algorithmic}
\end{algorithm}

\section{Numerical implementation details}
\label{sec.Numerical}
We discuss implementation details in this section.

\subsection{Numerical discretization}
\label{sec.numdis}

We discuss the discretization of functions defined on a two-dimensional domain; the extension to the three-dimensional case is analogous. 
We introduce an $N \times N$ uniform mesh (where $N \in \mathbb{N}$) with spacing $\Delta x = \Delta y = \frac{2\pi}{N}$ 
such that the grid is embedded in the square domain 
\[
\Omega = \left\{\left( -\pi + \frac{2i\pi}{N}, -\pi + \frac{2j\pi}{N} \right) \mid i,j \in \{1,2,\dots,N\}\right\} \subset \mathbb{R}^2,
\]
where $(i,j)$ denotes the integer grid indices. 
Let $u_{i,j}$ denote the value of the function $u$ at the grid point $(x_i, y_j) = \left( -\pi + \frac{2i\pi}{N}, -\pi + \frac{2j\pi}{N} \right) $. 
Partial derivatives of $u$ at $(x_i, y_j)$ are then approximated using central finite difference schemes.
\begin{align}
    \partial_1 u_{i,j}=\begin{cases}
		\frac{u_{i+1,j}-u_{i-1,j}}{2\Delta x} & 1 < i < N,\\
		\frac{u_{2,j}-u_{N,j}}{2\Delta x} & i =1,\\
		\frac{u_{1,j}-u_{N-1,j}}{2\Delta x} &i =N,
	\end{cases}
\qquad
    \partial_2 u_{i,j}=\begin{cases}
		\frac{u_{i,j+1}-u_{i,j-1}}{2\Delta y}& 1 < j < N,\\
		\frac{u_{j,2}-u_{j,N}}{2\Delta y} & j =1,\\
		\frac{u_{j,1}-u_{j,N-1}}{2\Delta y}& j =N.
	\end{cases}
\end{align}
The discretized gradient and normal vector are computed as
\begin{align*}
 &\nabla u_{i,j} = (\partial_1u_{i,j},\partial_2u_{i,j}), \  \mathbf{n}_{i,j}=\frac{ \nabla u_{i,j}}{|\nabla u_{i,j}|} \mbox{ with } |\nabla u_{i,j}|=\sqrt{(\partial_1u_{i,j})^2+ (\partial_2u_{i,j})^2}.
\end{align*}
Denote $\mathbf{n}_{i,j}=\left((n_1)_{i,j}, (n_2)_{i,j}\right)$. The discretized curvature is computed as
 \begin{align}
        (\kappa(u))_{i,j} =\partial_1(n_1)_{i,j} + \partial_2(n_2)_{i,j} .
        \label{kappa.normal}
 \end{align}

\subsection{Computation of the distance function}
\label{sec.distancefunction}
For a given point cloud $\mathcal{C} \subset \Omega$, let $ d(\bx) = \inf_{\mathbf{y}\in \mathcal{C}} |\mathbf{x} - \mathbf{y}| $ denote the Euclidean distance from a point $\bx$ to the point cloud $\mathcal{C}$. This distance function can be obtained by solving the Eikonal equation, which is given as
\begin{align}
    \begin{cases}
        | \nabla d|  = 1,
        \\
        d(\bx) = 0, \forall \bx \in \mathcal{C}.
    \end{cases}
\end{align}
To solve the equation, we adopt the first-order Lax-Friedrichs fast sweeping scheme \cite{kao2004lax}:
\begin{align}
\begin{cases}
d_{i,j}^{n+1} = 0, \quad \forall (i,j) \text{ such that } \mathbf{x}_{i,j} \in \mathcal{C}, \\
d_{i,j}^{n+1} = \min \left[d_{i,j}^{n} ,\frac{1}{\frac{\sigma_x}{\Delta x}+\frac{\sigma_y}{\Delta x}} \left( 1 - |\nabla d_{i,j}^{n}|  + \sigma_x \frac{d_{i+1,j}^{n} + d_{i-1,j}^{n}}{2\Delta x} + \sigma_y \frac{d_{i,j+1}^{n} + d_{i,j-1}^{n}}{2\Delta x} \right) \right] , \text{otherwise},
\label{lffunction}
\end{cases}
\end{align}
where $\Delta x = \frac{2\pi}{N}$ denotes the grid spacing, and $\sigma_x =\sigma_y = 1 $ represents the artificial viscosity (or directional weights) that weight the contributions in the $x$ and $y$ directions. We set $d^0 = 10^3 \cdot \mathbf{1}$ (all-ones matrix) as the initial guess, and the iteration stops at $\left|d^n - d^{n-1}\right| \leq \epsilon$. The converged $d^n$ is denoted as $d$, which is the distance function.

\subsection{Post-processing}
\label{sec.pp}
Since the output $u$ is a binary function, directly visualizing $u$ results in severe staircase artifacts. Prior to visualization, we transform $u$ into a level set function $\Theta^0$ according to the following relation:
\begin{align}
	\Theta^0 = G_{\tau}\ast [d(1-2u)].
\end{align}
We then reinitialize the level set function $\Theta$ by solving 
\begin{align}
    \frac{\partial\Theta}{\partial t} + \mathrm{sign}(\Theta)(1 - \left|\nabla \Theta\right|) = 0.
\end{align}
To solve this equation, we integrate the artificial viscosity term of the Lax-Friedrichs scheme with the second-order Runge-Kutta method following the framework in \cite{osher2003level}. Denote the artificial time step by $\Delta l$, the scheme is given by:
\begin{equation}
\begin{cases}
k_1 = -H_{i,j}^{1} \cdot \Delta l  , \\
\Theta^{n-1/2}_{i,j} = \Theta^{n-1}_{i,j} + k_1  , \\
k_2 = -H_{i,j}^{2} \cdot \Delta l  , \\
\Theta^n_{i,j} = \Theta^{n-1}_{i,j} + \frac{1}{2}(k_1 + k_2)  ,
\end{cases}
\end{equation}
where
\begin{align}
    \begin{cases}
    H_{i,j}^{1} =  H(\nabla \Theta^{n-1}_{i,j})  - \sigma\left( \frac{D^{+}_1\Theta^{n-1}_{i,j}-D^{-}_1\Theta^{n-1}_{i,j}}{2\Delta x} +  \frac{D^{+}_2\Theta^{n-1}_{i,j}-D^{-}_2\Theta^{n-1}_{i,j}}{2\Delta x}\right),
    \\
    H_{i,j}^{2} =  H(\nabla \Theta^{n-1/2}_{i,j})  - \sigma\left(\frac{D^{+}_1\Theta^{n-1/2}_{i,j}-D^{-}_1\Theta^{n-1/2}_{i,j}}{2\Delta x} +  \frac{D^{+}_2\Theta^{n-1/2}_{i,j}-D^{-}_2\Theta^{n-1/2}_{i,j}}{2\Delta x}\right) .
    \end{cases}
\end{align}
The artificial viscosity coefficient $\sigma$ and pseudo-time $\Delta l$ are typically set to $\sigma = 1$, $\Delta l = 10^{-4}$, and $ \Delta x = \frac{2\pi}{N}$, respectively. To ensure numerical stability and computational efficiency, we impose a maximum number of iterations as the stopping criterion such that $n \leq n_{\text{max}}$. 
\section{Numerical experiments}
\label{sec.experiment}
In this section, we demonstrate the effectiveness of the proposed algorithm through two-dimensional and three-dimensional examples. We compare our method with (i) the Iterative Thresholding Method (ITM) \cite{wang2021iterative}, which solves model (\ref{eq.model.dis}), and (ii) the operator-splitting method \cite{he2020curvature}, which solves model (\ref{eq.ener.0}). We denote this method by OSM+Curvature. All experiments are conducted in MATLAB on a desktop with \textbf{CPU}: AMD Ryzen 9 7900X and 16GB of RAM
 
\subsection{Two-dimensional examples}
\label{sec.experiment2D} 
All two-dimensional numerical experiments are conducted on a Cartesian grid of size $100 \times 100$. For consistency, all point cloud data are defined on the domain $\mathcal{H} = [-\pi, \pi] \times [-\pi, \pi] \subset \mathbb{R}^2$. Unless otherwise specified, the default parameter values are set as follows: $\tau_0 = 0.04$, $\Delta t = 800$, $k = 32$, $\eta = 0.001{\Delta x}^{\beta}$, $\alpha = 4$, $\beta = 4$, $\gamma = 0.04$, and $\alpha_0 = 0.99$. 
We set the number of iterations in reinitalization as $n_{max} =5$. 
The initial condition $u_0$ is defined as a centered box function such that:
\begin{align*}
    \begin{cases}
        u_0(\bx) = 1,\  \bx \in \mathcal{D},\\
        u_0(\bx) = 0, \ \bx \notin \mathcal{D},
    \end{cases}
\end{align*}
where the spatial domain $\mathcal{D} =\left[ -\frac{3\pi}{4}, \frac{3\pi}{4} \right] \times \left[ -\frac{3\pi}{4}, \frac{3\pi}{4} \right]$.
We update the iteration parameter as $\tau_{\text{new}} = \frac{\tau}{2}$ if the $L_1$-norm of the update satisfies $\| u_{n+1} - u_n \|_1 = 0$. The iterative process terminates when $\tau_{\text{new}} \leq \frac{\tau_0}{m}$, where $m > 1$, specifying the termination threshold. In the choice of $\eta$, we include a factor of $\Delta x^{\beta}$. We note that $\mathbf{n}_{i,j}$ is a unit vector by definition, and the magnitude of the elements of $\mathbf{n}$ is of $O(1)$. Given that $u$ is a binary function, $\pm (n_{i\pm1,j}-n_{i,j})$ and $\pm (n_{i,j\pm1}-n_{i,j})$ also have magnitude of $O(1)$. According to (\ref{kappa.normal}), the computed curvature has a magnitude of $O(\Delta x^{-\beta})$. In order to balance the scales of the two terms in (\ref{eq.ener.0}), the scaling factor ${\Delta x}^\beta$ is introduced for parameter $\eta$.

\subsubsection{General performance}
\label{section7.1}
\begin{figure}[htbp]
    \centering
    \begin{minipage}{0.32\linewidth}
        \centering
        \subfigure{\includegraphics[trim=45 0 40 20, width=\linewidth]{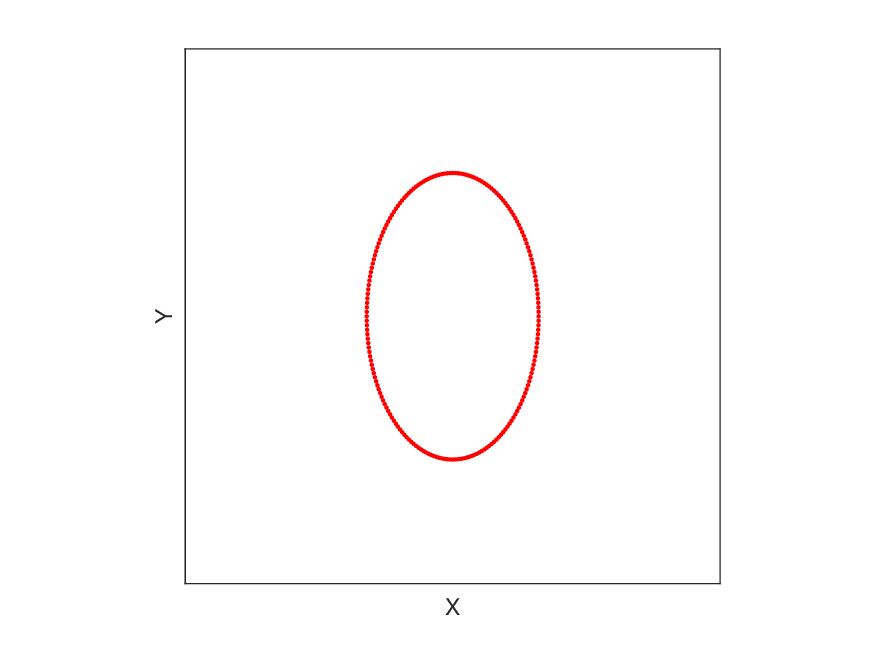}}\\
        \subfigure{\includegraphics[trim=45 0 40 20, width=\linewidth]{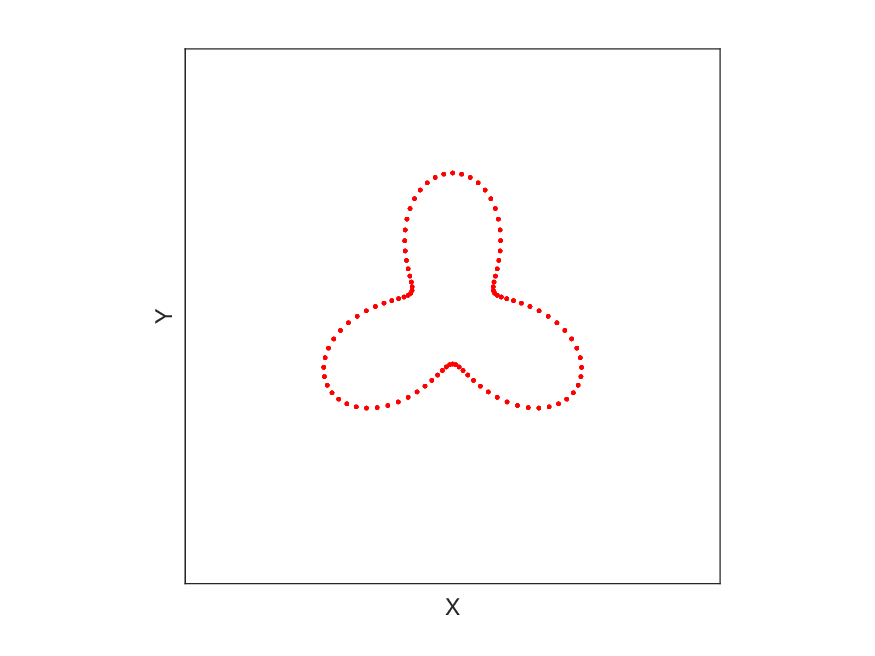}}\\
        \subfigure{\includegraphics[trim=45 0 40 20, width=\linewidth]{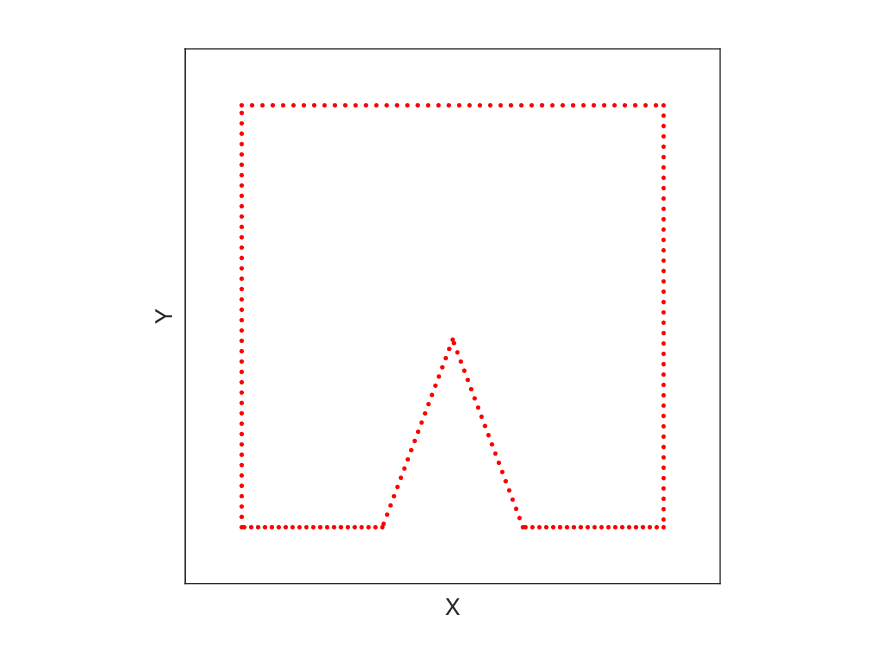}}\\
        \subfigure{\includegraphics[trim=45 0 40 20, width=\linewidth]{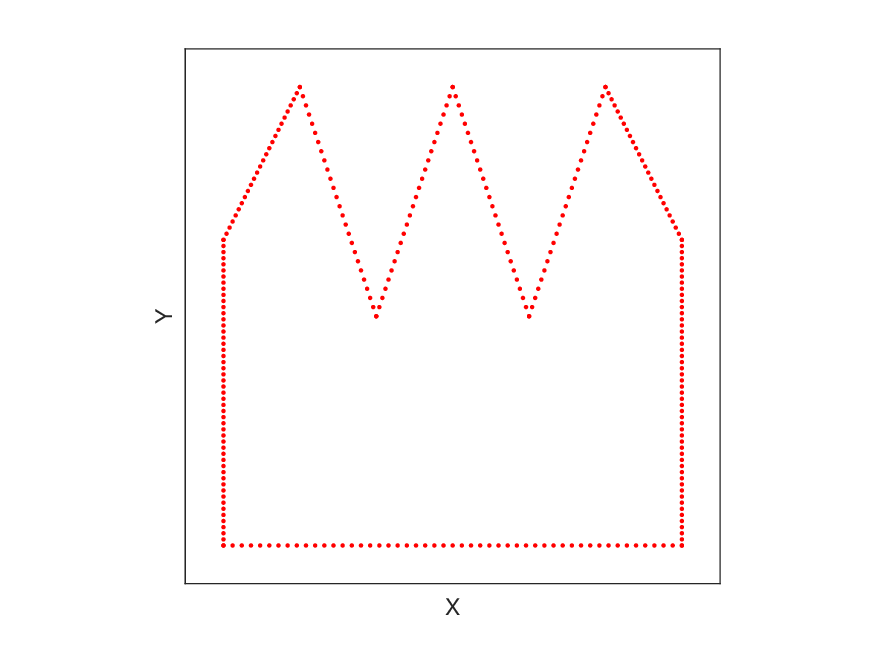}}
    \end{minipage}
    \hfill
    \begin{minipage}{0.32\linewidth}
        \centering
        \subfigure{\includegraphics[trim=45 0 40 20, width=\linewidth]{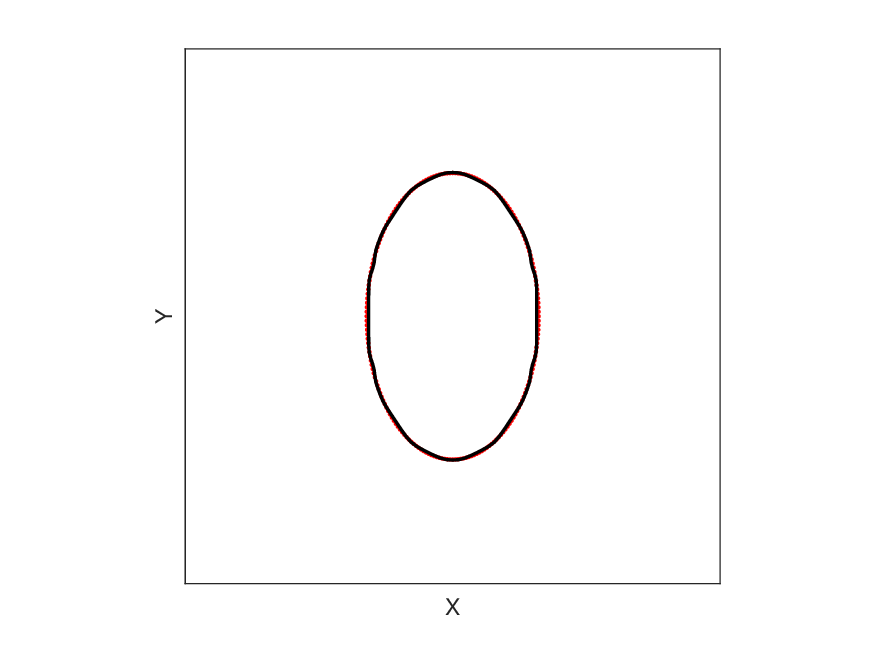}}\\
        \subfigure{\includegraphics[trim=45 0 40 20, width=\linewidth]{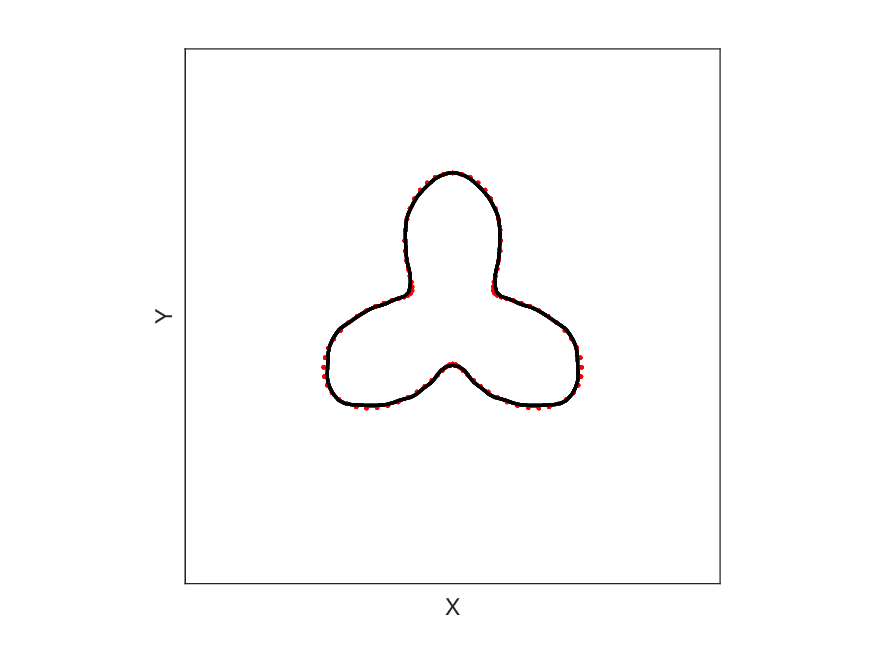}}\\
        \subfigure{\includegraphics[trim=45 0 40 20, width=\linewidth]{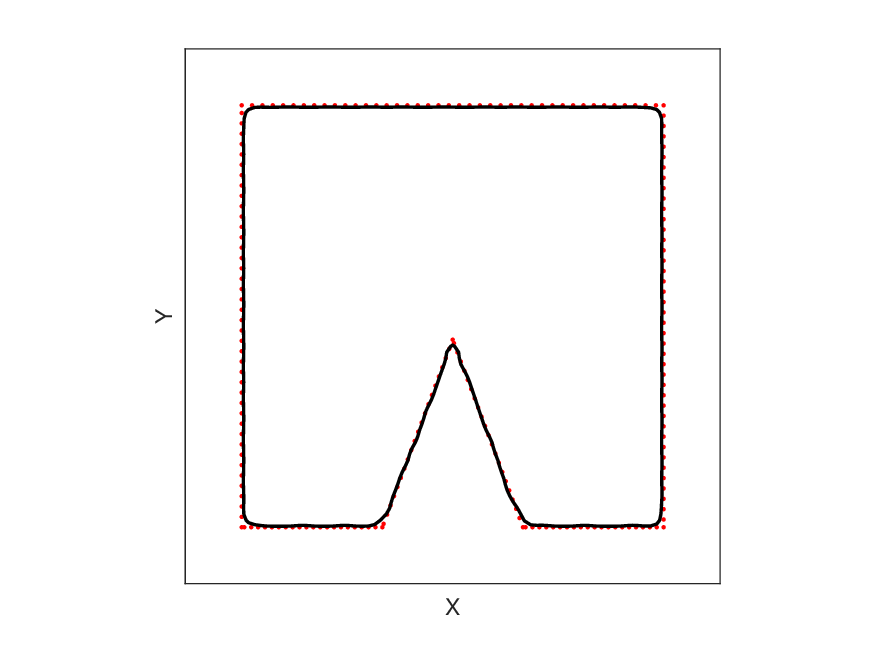}}\\
        \subfigure{\includegraphics[trim=45 0 40 20, width=\linewidth]{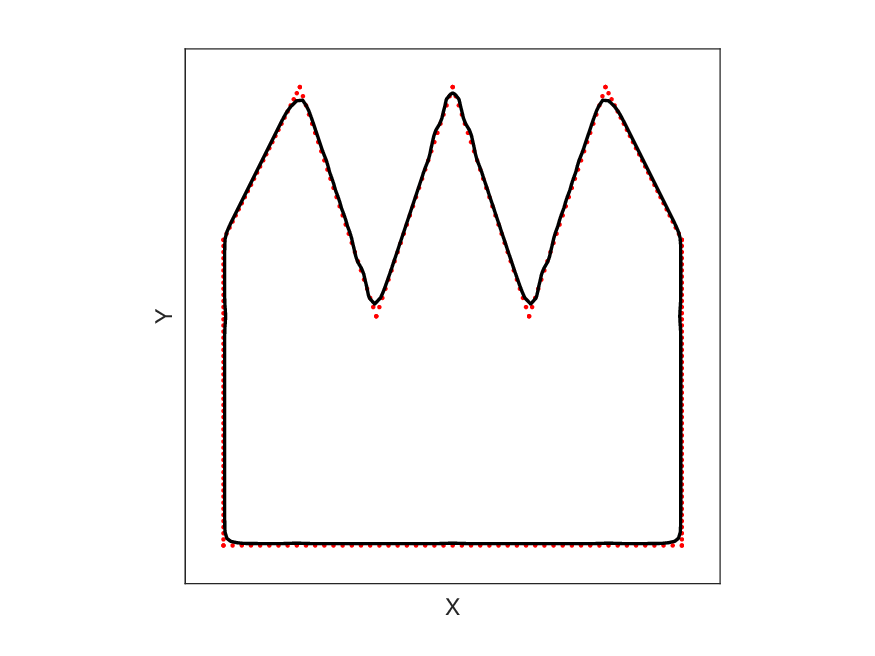}}
    \end{minipage}
    \hfill
    \begin{minipage}{0.33\linewidth}
        \centering
        \subfigure{\includegraphics[trim=45 0 40 20, width=\linewidth]{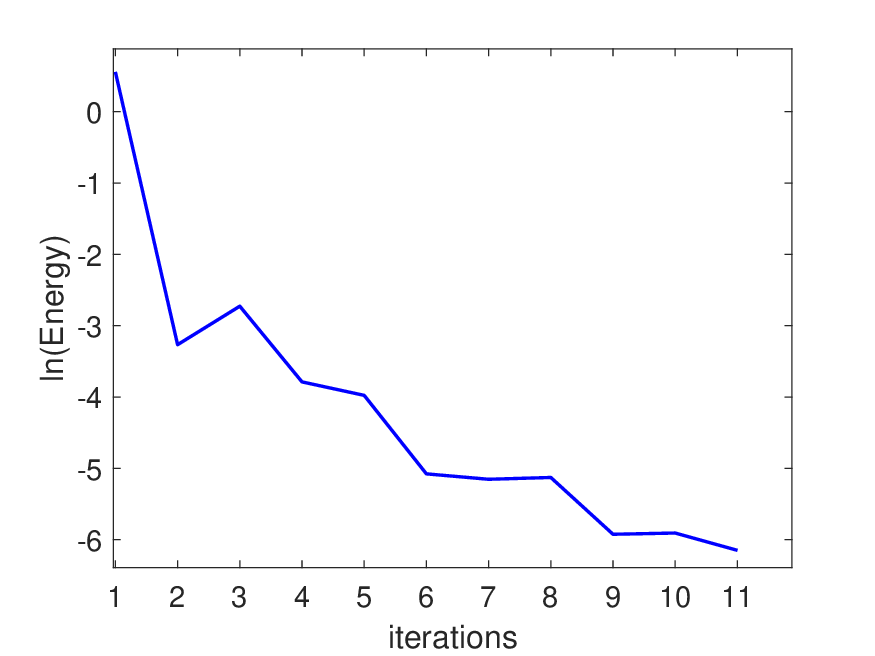}}\\
        \subfigure{\includegraphics[trim=45 0 40 20, width=\linewidth]{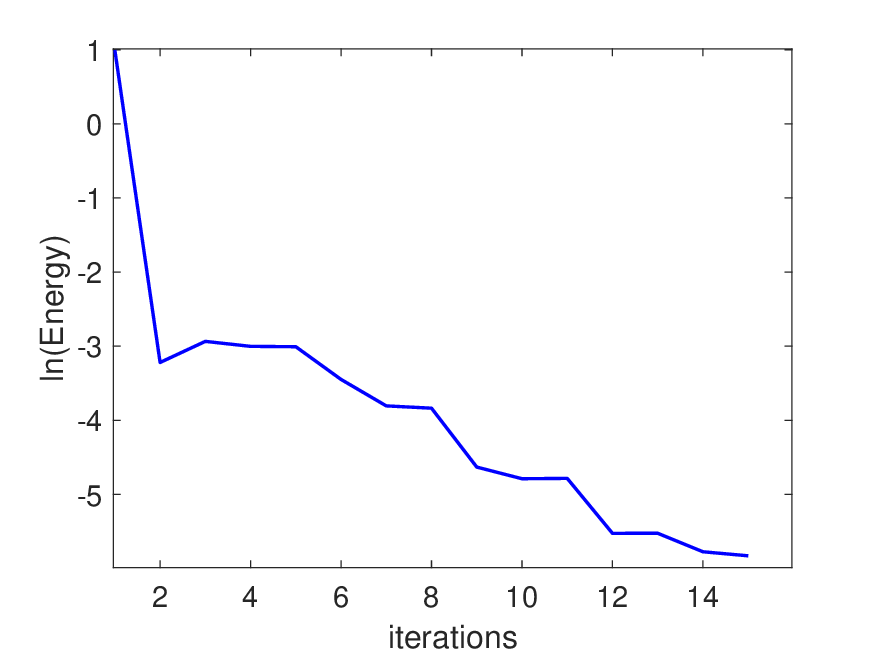}}\\
        \subfigure{\includegraphics[trim=45 0 40 20, width=\linewidth]{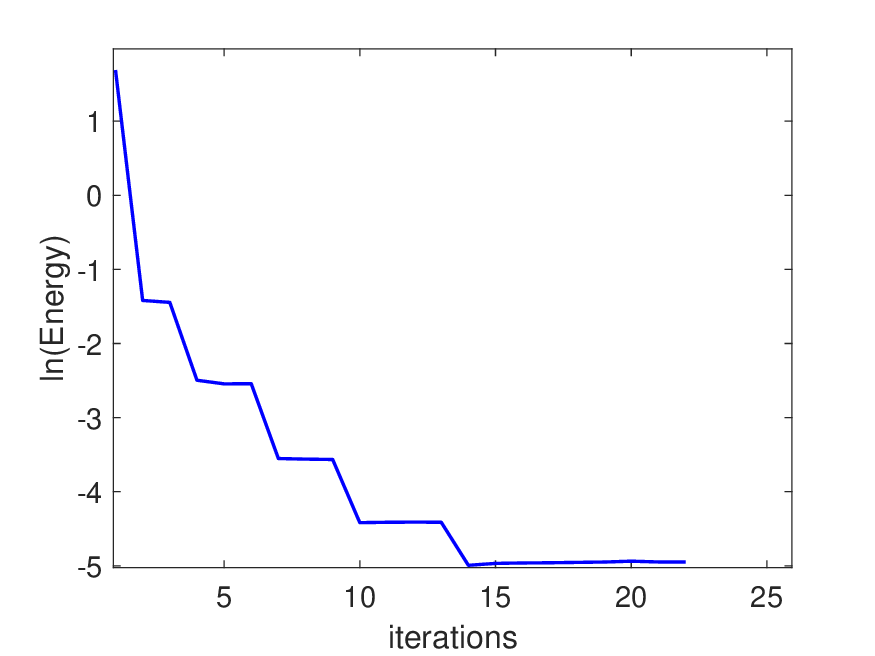}}\\
        \subfigure{\includegraphics[trim=45 0 40 20, width=\linewidth]{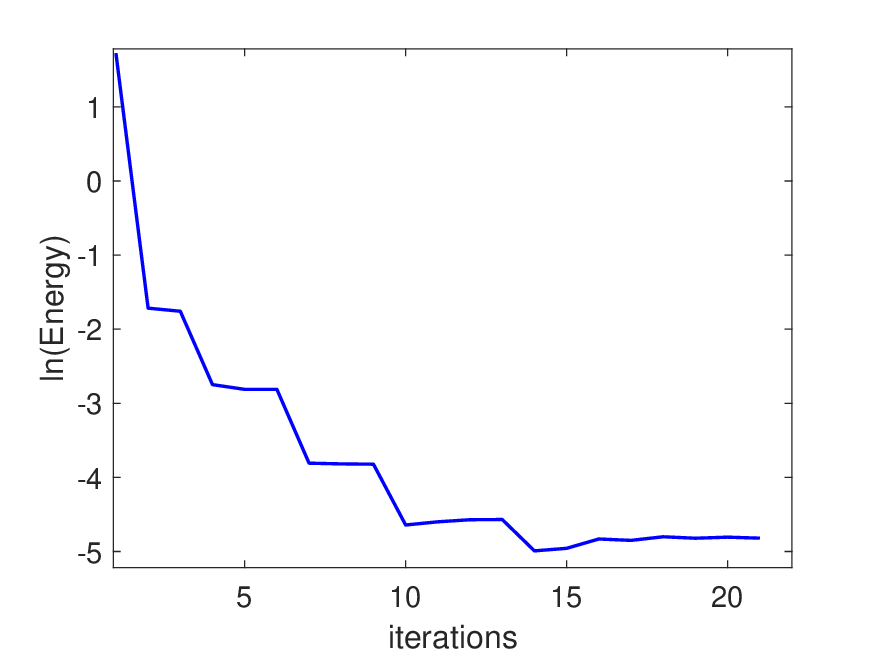}}
    \end{minipage}

    \caption{(Two-dimensional example.) General performance of the proposed method. Column 1: point cloud. Column 2: reconstructed curve by the proposed method. Column 3: Energy history.}
    \label{fig:2DExampleEN}
\end{figure}
We first apply the proposed algorithm to four two-dimensional point clouds. The first two point clouds are sampled from smooth curves, while the last two point clouds are sampled from curves with sharp corners and angles. The point clouds and the reconstructed curves are shown in the first and second column of Figure \ref{fig:2DExampleEN}. In this set of experiments, the proposed algorithm reconstructs the underlying curves well. Details such as corners and nonconvex features are preserved.

To demonstrate the convergent behavior of our algorithm, we present the energy history of the four examples in the third column of Figure \ref{fig:2DExampleEN}. Our algorithm effectively solves (\ref{eq.ener.0}), and converges within 25 iterations.

\subsubsection{Comparison with other methods}
\label{section7.2}

We compare our method with ITM and OSM+Curvature on four point clouds. The point clouds and results are shown in Figure \ref{fig:Experiment-1}. In each figure, the blue curve represents the ITM method, the green curve represents the OSM+Curvature method, and the black curve represents our proposed method. Red points are point cloud points. In this experiment, we set $\tau=0.04$ for ITM and $\eta=2$ for OSM+Curvature. Here, $\eta$ is the weight parameter of the curvature term in the OSM+Curvature method.

In this comparison, ITM fails to reconstruct sharp angles. This is because ITM solves (\ref{eq.model.dis}), which is based solely on the distance function. To overcome this difficulty, our proposed method and OSM+Curvature solve (\ref{eq.ener.0}), which uses a curvature term as a regularizer, enabling it to capture sharp angles and thus effectively reconstruct both convex and concave features.

Although OSM+Curvature and the proposed method solve the same model, the results produced by OSM+Curvature tends to be over-smoothed, as shown in Figure \ref{fig:Experiment-1} (c) and (d). This is because OSM+Curvature solving this problem by the level set method, for which reinitialization is needed to make the algorithm stable. However, reinitialization introduces additional smoothing effects, which makes the result over-smooth. In this comparison, our method provides the best results and better recovers sharp corners.
\begin{figure}[htbp]
\centering
    \subfigure[ ]{
    \includegraphics[trim = 30 30 25 30,width=0.35\linewidth]{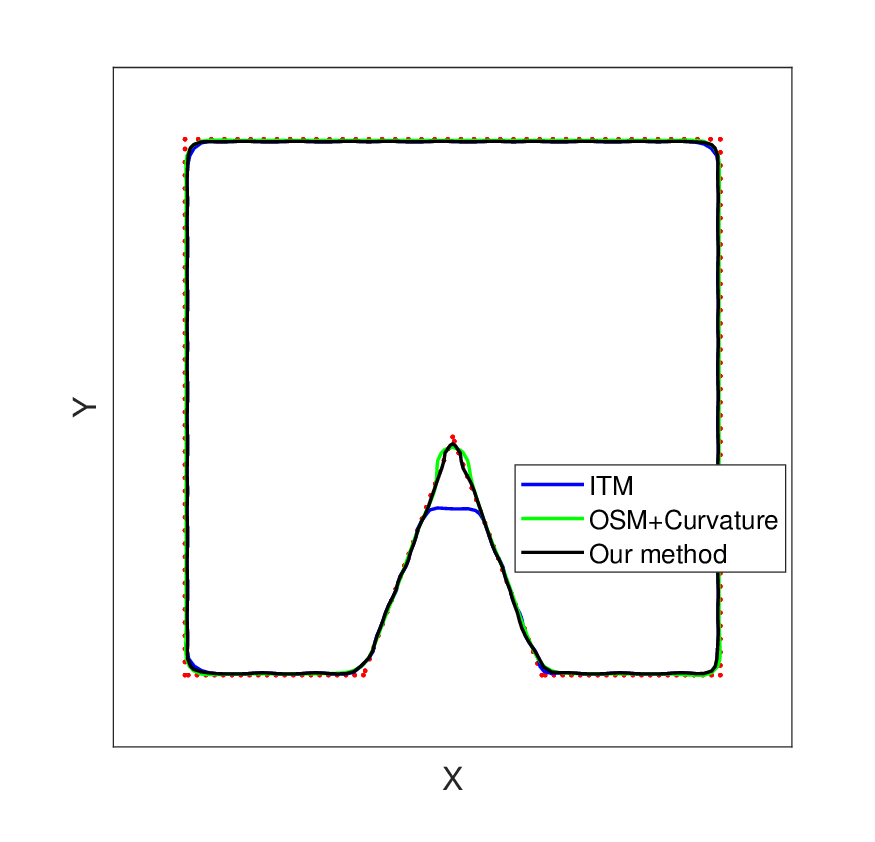}
    }
    \subfigure[ ]{
    \centering
    \includegraphics[trim = 30 30 25 30,width=0.35\linewidth]{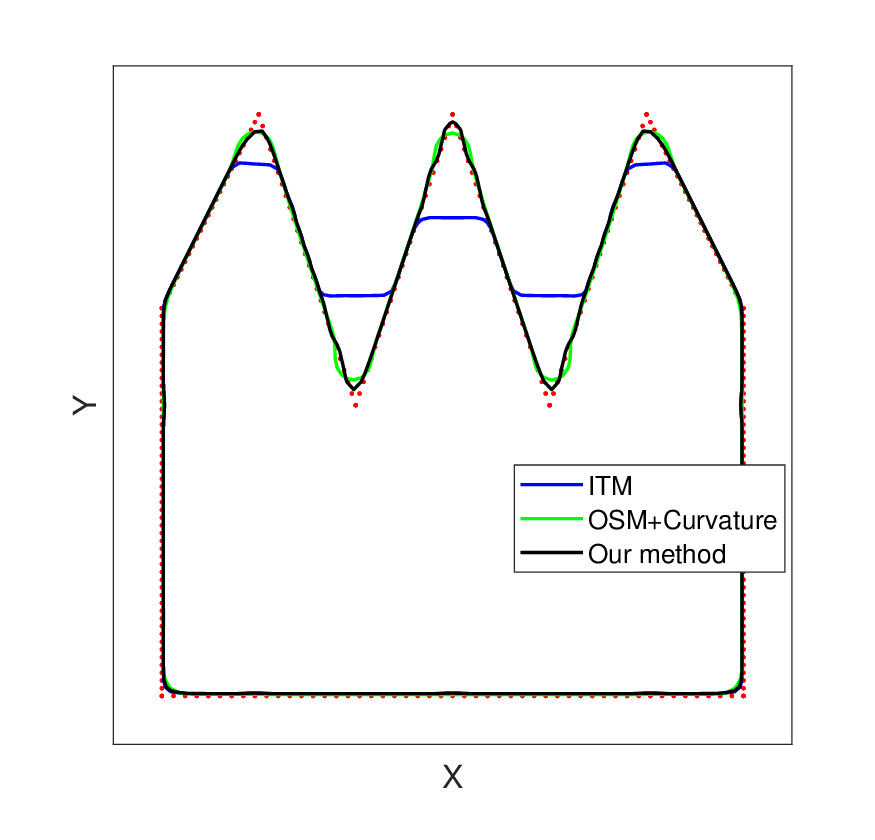}
    }
    \\
    \subfigure[ ]{
    \centering
    \includegraphics[trim =30 30 25 30,width=0.35\linewidth]{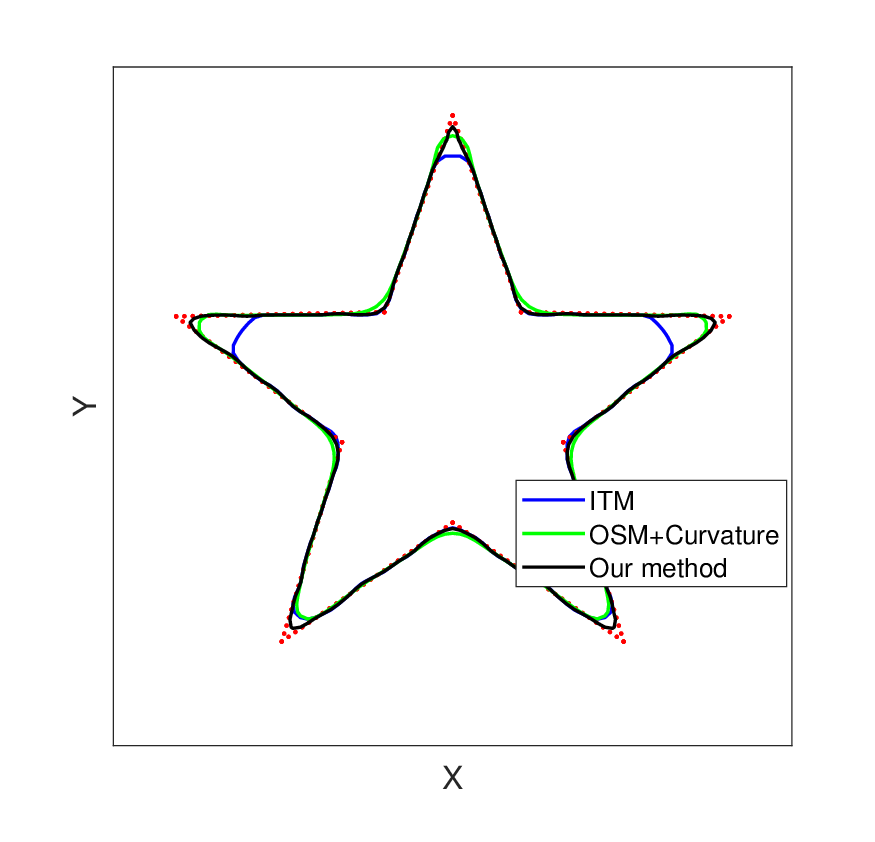}
    }
    \subfigure[ ]{
    \centering
    \includegraphics[trim = 30 30 25 30,width=0.35\linewidth]{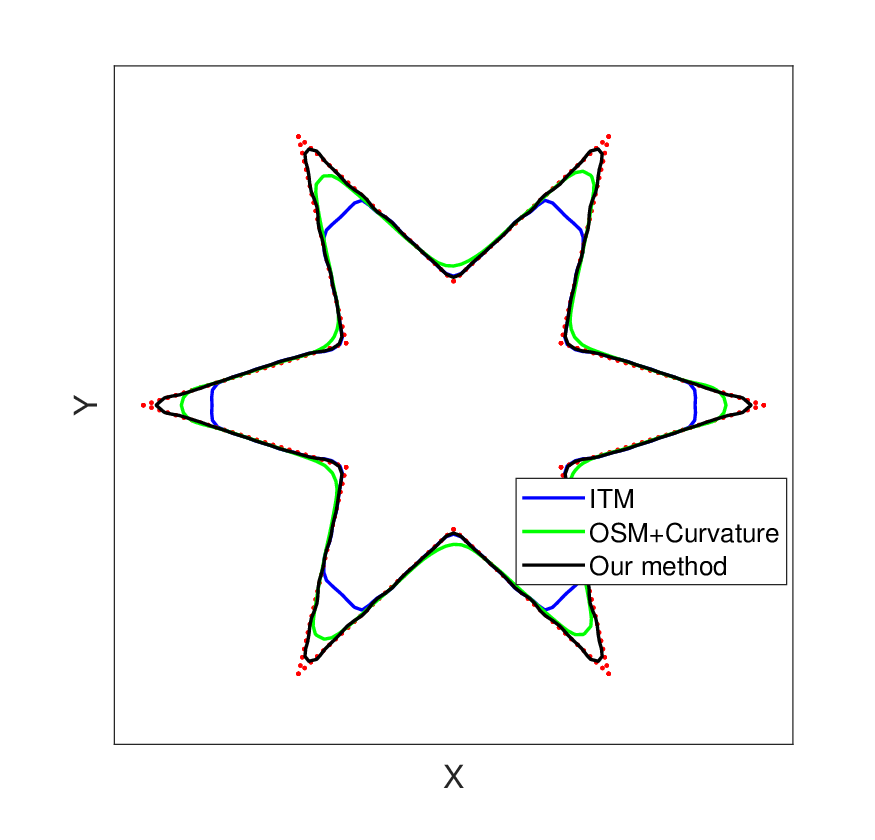}
    }
     \caption{(Two-dimensional example.) Comparison of ITM, OSM+Curvature and our method.
     }
    \label{fig:Experiment-1}
\end{figure}

\begin{table}[t!]
\begin{center}
    
\begin{tabular}{|c| c| c| c| c|} 

 \hline
  & (a) & (b) & (c) & (d)\\ [0.5ex] 
 \hline
 ITM & 0.055s & 0.050s & 0.040s & 0.043s \\ 
 \hline
 OSM+curvature & 21.31s & 23.10s & 4.12s & 4.40s\\
 \hline
 Our method & 0.14s &  0.14s & 0.15s & 0.16s \\
 \hline
\end{tabular}
\end{center}
\caption{(Two-dimensional example.) Required CPU time to compute results in Figure \ref{fig:Experiment-1}.}
\label{table:time1}
\end{table}
We then compare the efficiency of all three methods. Note that OSM+Curvature requires a large fixed number of iterations to reconstruct nonconvex sharp angles. Therefore, for Figure \ref{fig:Experiment-1} (a) and (b), we use 300 iterations for OSM+Curvature. The CPU time required to compute results in Figure \ref{fig:Experiment-1} are shown in Table \ref{table:time1}. ITM is the fastest, and the proposed algorithm is only slightly slower, with CPU time on the same scale. Compared with ITM and our method, OSM+Curvature is much slower. 

In conclusion, our method achieves performance comparable to, or even better than, that of the OSM+Curvature method, while exhibiting a significantly reduced computational cost, which is comparable to ITM.
\subsubsection{Effects of parameter $\alpha$ and $\beta$ }
\label{section7.3}
\begin{figure}[htbp]
    \centering
    \subfigure[ ]{
    \includegraphics[trim = 30 30 25 30,width=0.35\linewidth]{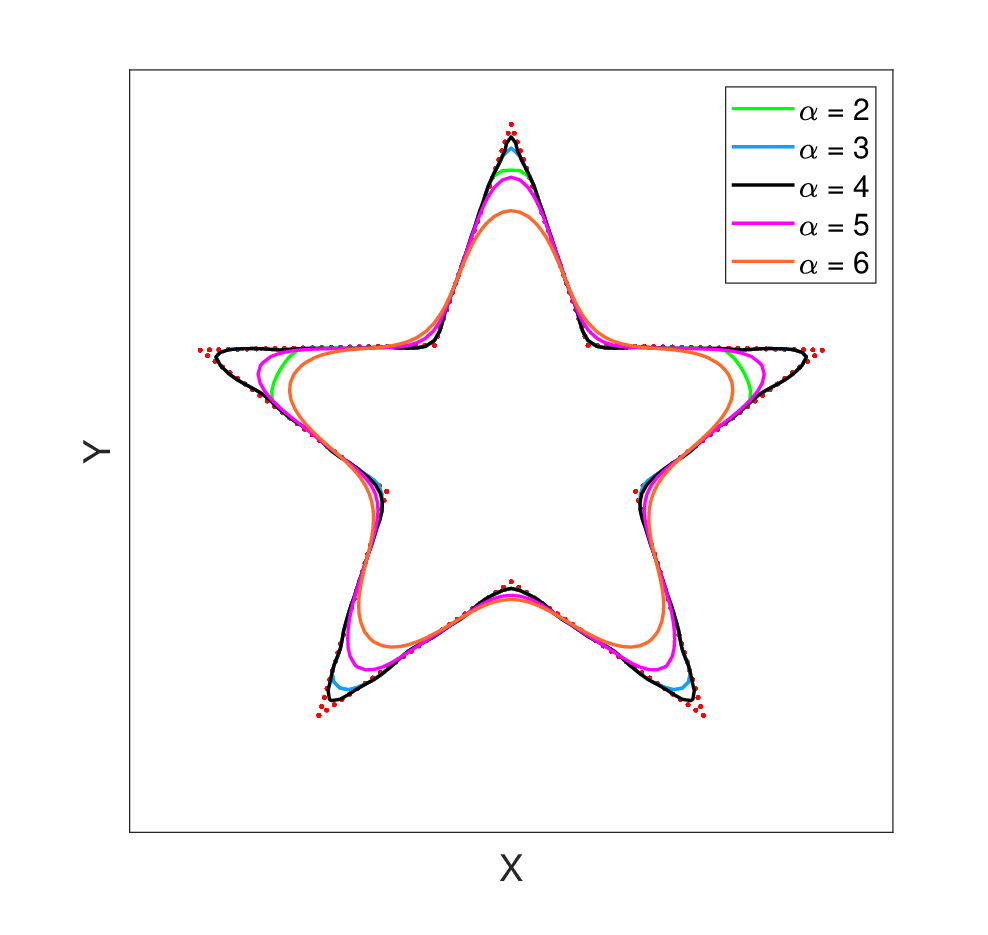}
    }
    \subfigure[ ]{
    \includegraphics[trim = 30 30 25 30,width=0.35\linewidth]{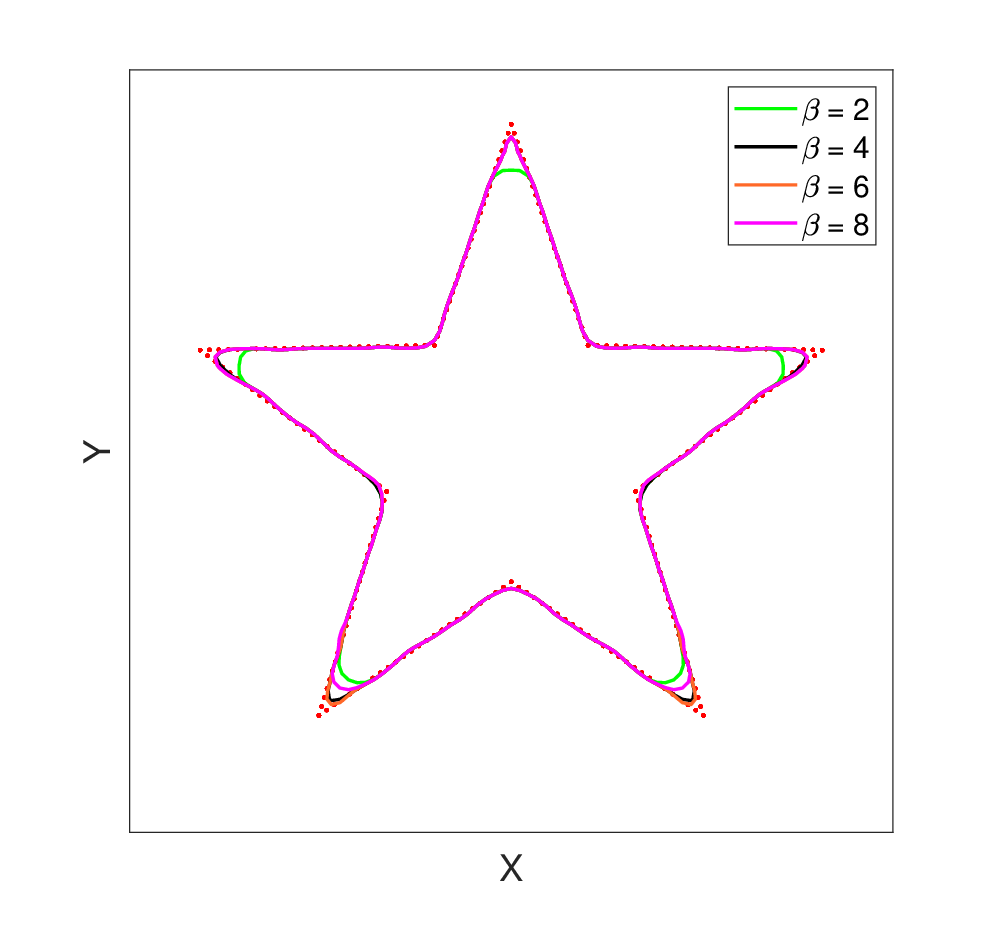}
    }
    
    \caption{(Two-dimensional example.) Reconstruction of five-pointed star by our method with various $\alpha$ and $\beta$.}
    \label{fig:Experiment-8}
\end{figure}
In (\ref{eq.ener.0}), $\alpha$ and $\beta$ are exponents of the distance term and curvature term, respectively. In this section, we conduct an ablation study to investigate the effects of $\alpha$ and $\beta$.

Figure \ref{fig:Experiment-8} (a) depicts the results for \(\beta = 4\) and \(\alpha = 2, 3, 4, 5, 6\). It indicates that our algorithm with \(\alpha = 4\) achieves the best reconstruction in the five-pointed star result. For \(\alpha = 2, 5, 6\), the algorithm cannot converge to sharp angles. For \(\alpha = 3\) and \(4\), the algorithm successfully recovers the sharp features. However, the performance for \(\alpha = 4\) is slightly better than that for \(\alpha = 3\).

Figure \ref{fig:Experiment-8} (b) depicts the results for \(\alpha = 4\) and \(\beta = 2, 4, 6, 8\). Since the parameter $\beta$ influences the scaling factor of $\eta$, we set the corresponding values as $\eta = 0.004{\Delta x}^2$, $\eta = 0.004{\Delta x}^4$, $\eta = 0.004{\Delta x}^6$, and $\eta = 0.004{\Delta x}^8$, respectively. According to Figure \ref{fig:Experiment-8} (b), when \(\beta = 2\), the algorithm fails to converge at the sharp corners of the five-pointed star. A higher exponent increases the influence of \(\kappa^{\beta}\) in the energy near sharp features. However, when \(\beta = 6, 8\), there is no significant improvement compared to \(\beta = 4\). It appears that increasing the weight of \(\kappa^{\beta}\) beyond this point yields diminishing returns.

In conclusion, the parameter pair \(\alpha =4,\beta = 4\) is a reasonable choice of parameters.

\subsubsection{Effects of $\eta$}
\label{section5.2}
\begin{figure}[t!]
    \centering
    
    \includegraphics[trim = 30 30 25 30,width=0.35\linewidth]{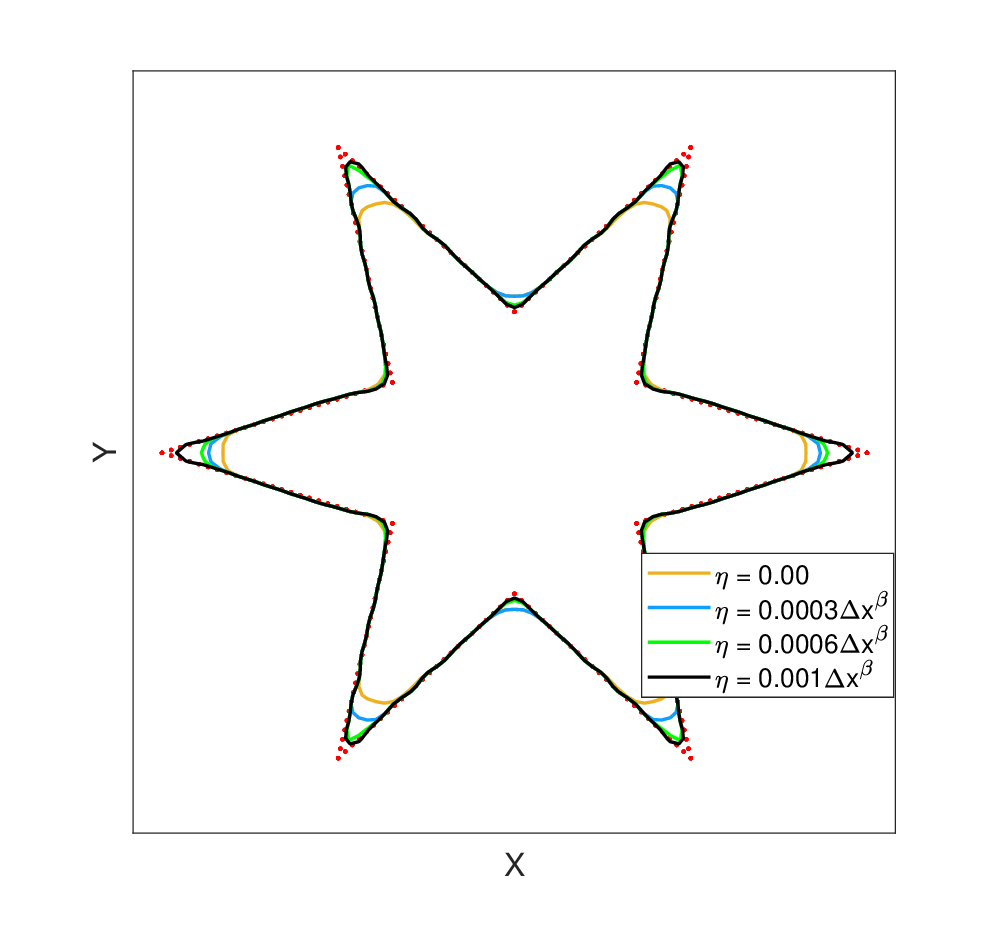}
    
    \caption{(Two-dimensional example.) Reconstruction of six-pointed star with various parameter $\eta$.}
    \label{fig:Experiment-5}
\end{figure}
According to (\ref{eq.ener.0}), the parameter $\eta$ controls the strength of the curvature term. We conduct four experiments using different values of this parameter, including $\eta = 0,\eta = 0.0003{\Delta x}^\beta,\eta = 0.0006{\Delta x}^\beta,\eta = 0.001{\Delta x}^\beta$ with  $\beta =4$. The experimental results are shown in Figure \ref{fig:Experiment-5}. The results demonstrate that, as $\eta$ increases, sharp corners are reconstructed more accurately. Thus, in our experiments, we use $\eta=0.001\Delta x^{\beta}$.

\begin{figure}[t!]
    \centering
    \subfigure{\includegraphics[trim=25 25 25 25,clip,width=0.32\linewidth]{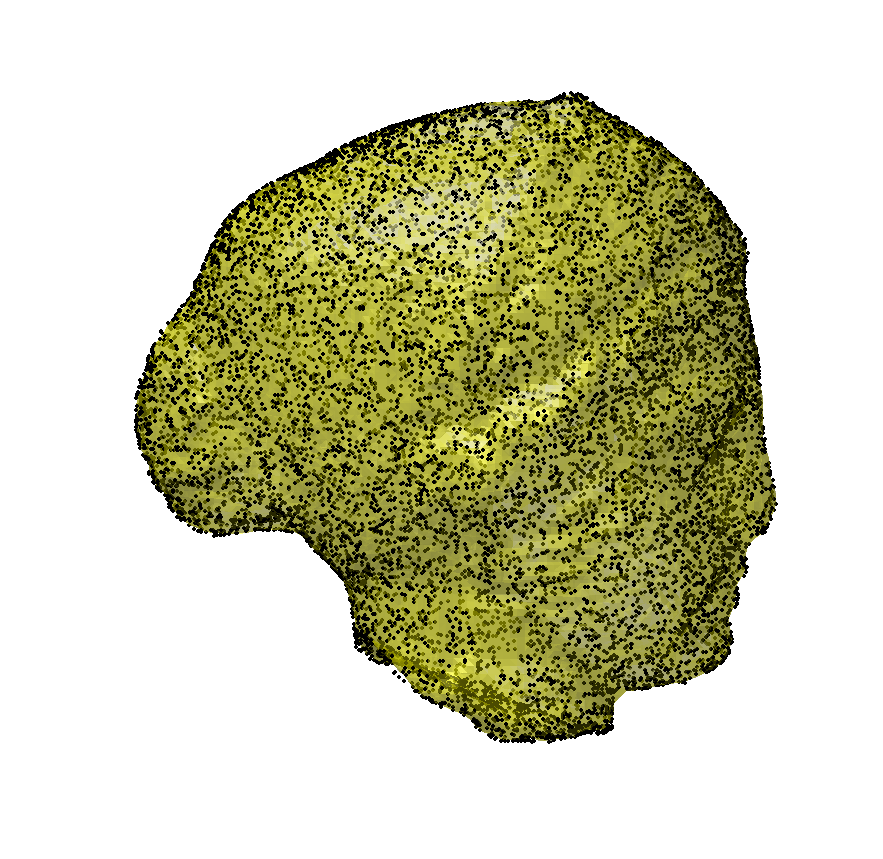}}%
    \hfill
    \subfigure{\includegraphics[trim=25 25 25 25,clip,width=0.32\linewidth]{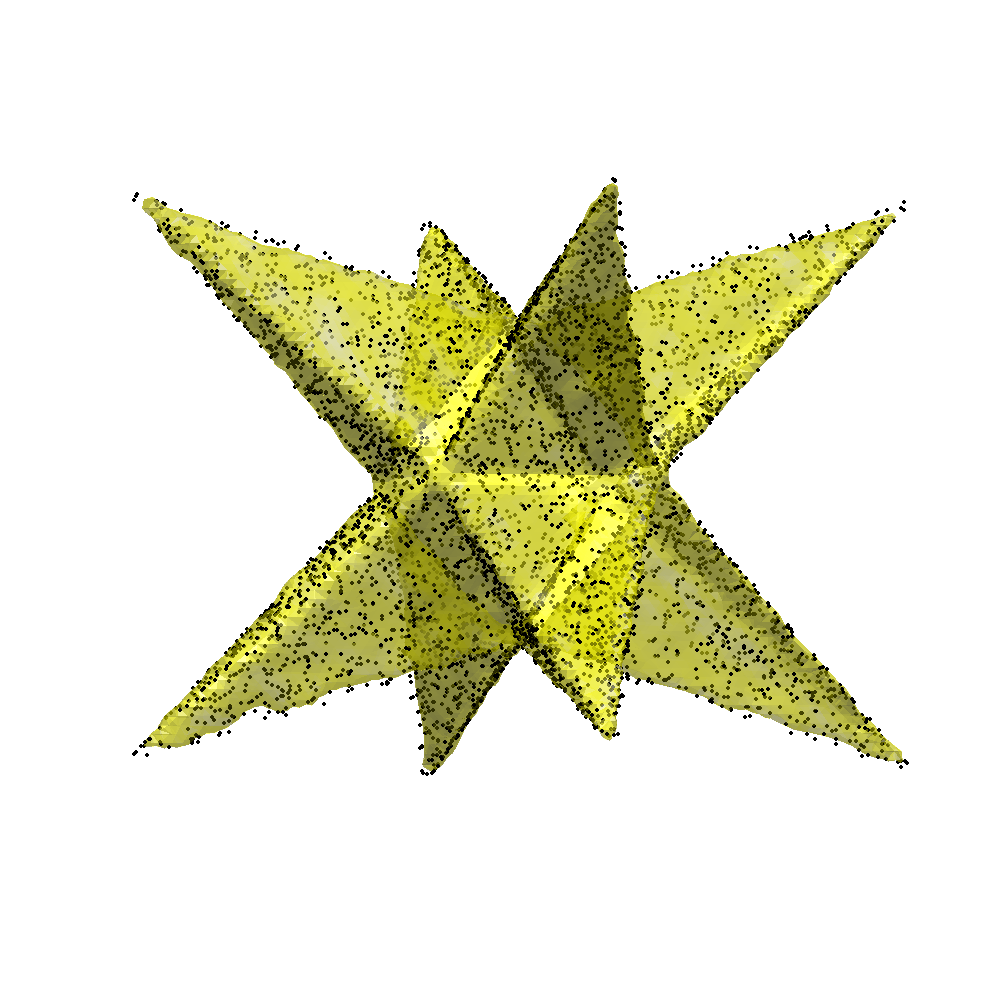}}%
    \hfill
    \subfigure{\includegraphics[trim=25 25 25 25,clip,width=0.32\linewidth]{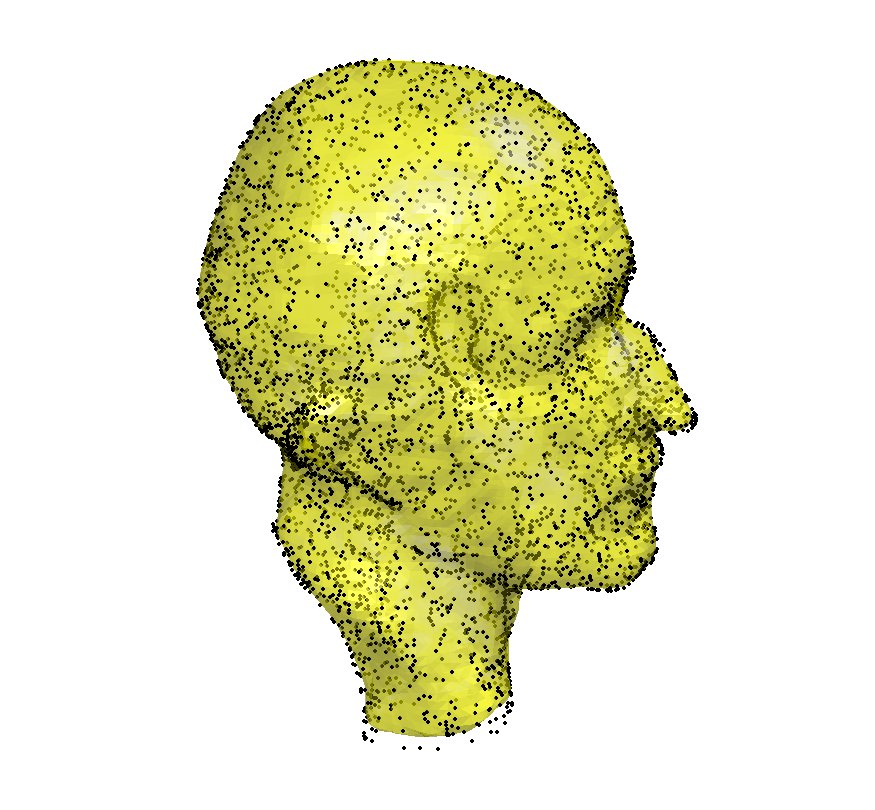}}%

    \vspace{0.5em}

    \subfigure{\includegraphics[trim=100 70 70 70,clip,width=0.32\linewidth]{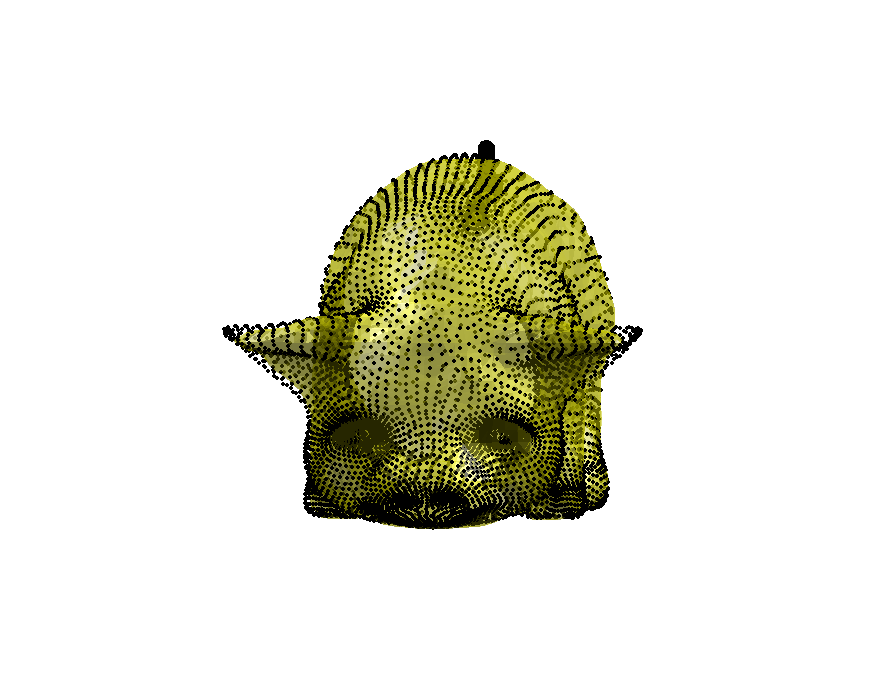}}%
    \hfill
    \subfigure{\includegraphics[trim=25 25 25 25,clip,width=0.32\linewidth]{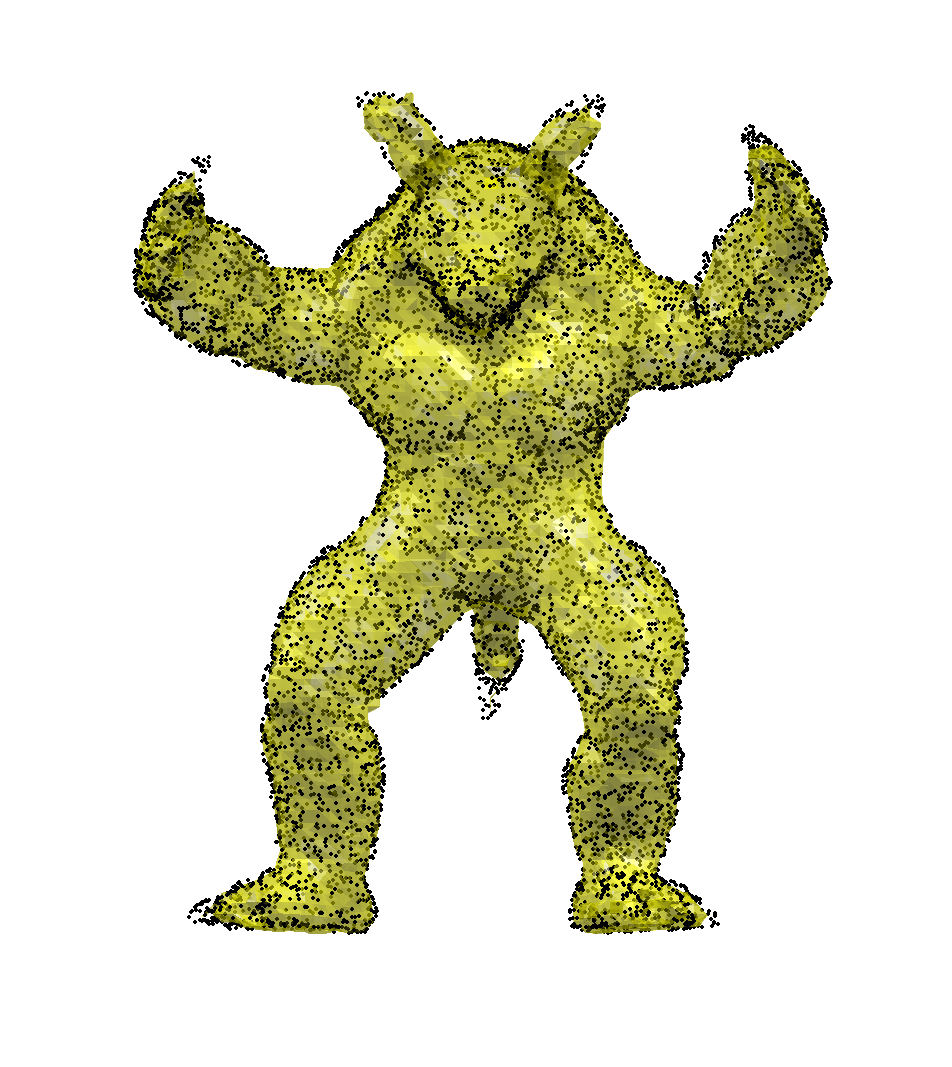}}%
    \hfill
    \subfigure{\includegraphics[trim=25 25 25 25,clip,width=0.32\linewidth]{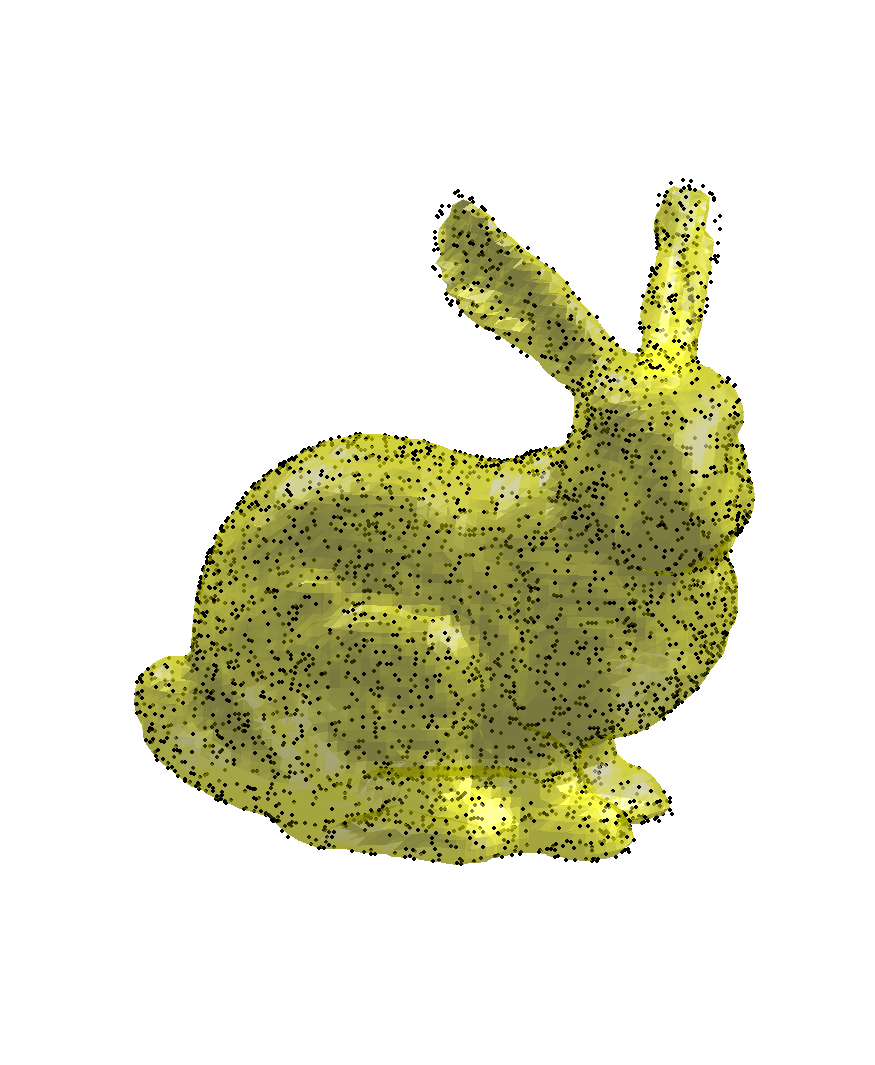}}%

    \caption{(Three-dimensional example.) General performance of our method.}
    \label{fig:Example3D}
\end{figure}

\subsection{Three-dimensional examples}
\label{sec.experiment3D}
We consider three three-dimensional point cloud datasets and test the proposed algorithm. 
All experiments are conducted in the domain $\mathcal{H} = [-\pi, \pi] \times [-\pi, \pi] \times [-\pi, \pi] \subset \mathbb{R}^3$. The initial condition $u_0$ is defined as a centered cube function such that:
\begin{align*}
    \begin{cases}
        u_0(\bx) = 1, \ \bx \in \mathcal{D},\\
        u_0(\bx) = 0, \ \bx \notin \mathcal{D},
    \end{cases}
\end{align*}
where the spatial domain $\mathcal{D} = [-\frac{\pi}{2}, \frac{\pi}{2}] \times [-\frac{\pi}{2}, \frac{\pi}{2}] \times [-\frac{\pi}{2}, \frac{\pi}{2}] \subset \mathbb{R}^3$.
Unless otherwise specified, we use $\tau_0 = 0.01$, $\Delta t = 800$, $\eta = 0.004{\Delta x}^{\beta}$, $\alpha = 4$, $\beta = 4$, $\gamma = 0.02$, and $\alpha_0 = 0.5$.
We set the number of iterations in reinitialization to be $n_{max} =25$.

\subsubsection{General performance}
To begin, we evaluate the performance of our method on several general surface reconstruction tasks. For this experiment, we utilize six point clouds from \cite{Stanford3D,dfaust:CVPR:2017,jacobson2020common} and employ a Cartesian grid with a resolution of $200\times 200 \times 200$. As illustrated in Figure \ref{fig:Example3D}, our approach successfully reconstructs intricate surface details, including the ears of the pig and rabbit, and the tail of the armadillo. Together, these findings demonstrate the algorithm's robustness in preserving complex geometric features and underscore the effectiveness of the proposed method.
\subsubsection{Comparison with other methods}
We compare our method with OSM+Curvature and ITM on three point clouds shown in Figure \ref{fig:Original point-cloud}, including (a) an icositetrahedron surface formed by attaching a square pyramid to each face of a cube, (b) a dumbbell-shaped surface and (c) a V-shaped surface. In this comparison, we discretize the domain by Cartesian grids of size $80 \times 80 \times 80$.

\begin{figure}[t!]
    \centering
    \subfigure[]{
    \centering
    \includegraphics[width=0.28\linewidth]{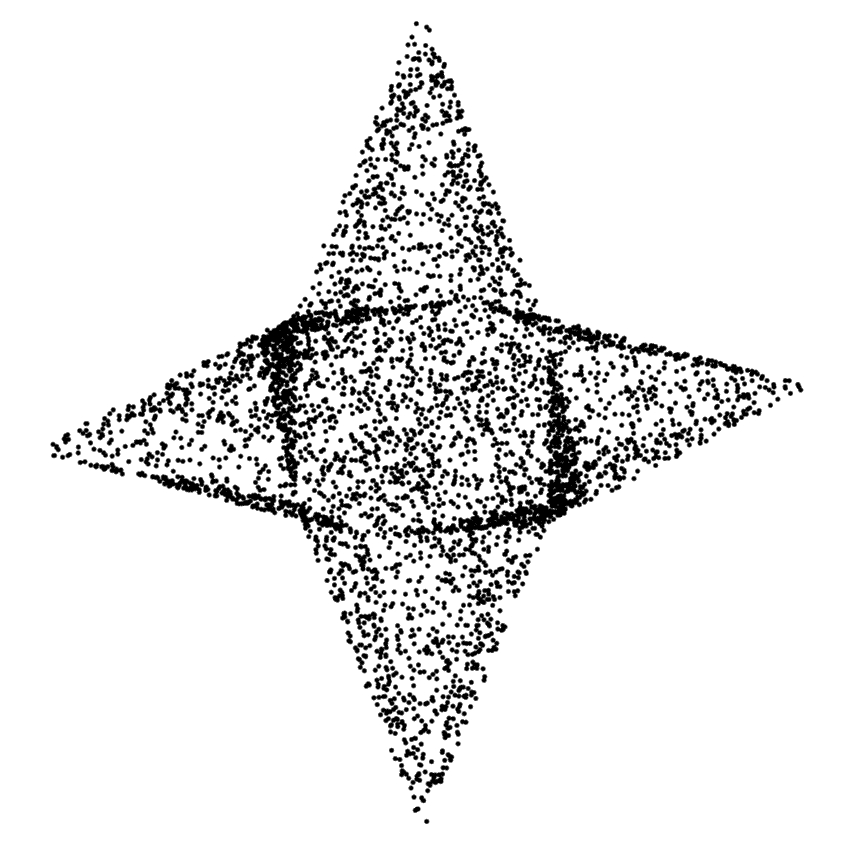}
    }
    \subfigure[]{
    \centering
    \includegraphics[width=0.22\linewidth]{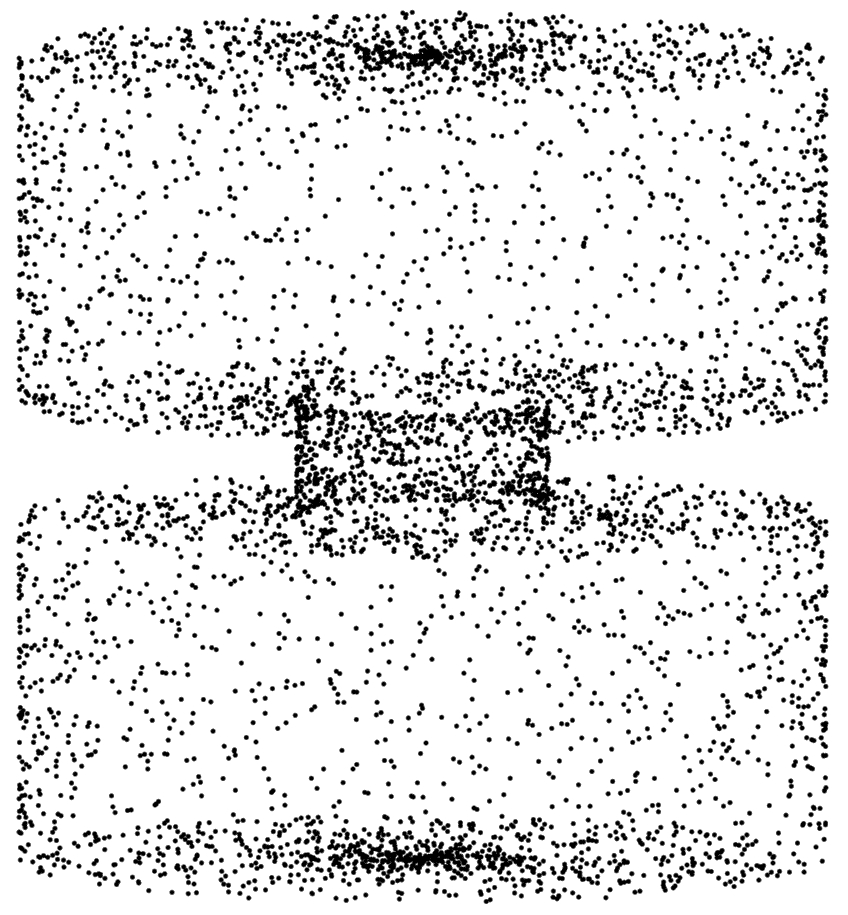}
    }
    \centering
    \subfigure[]{
    \includegraphics[width=0.25\linewidth]{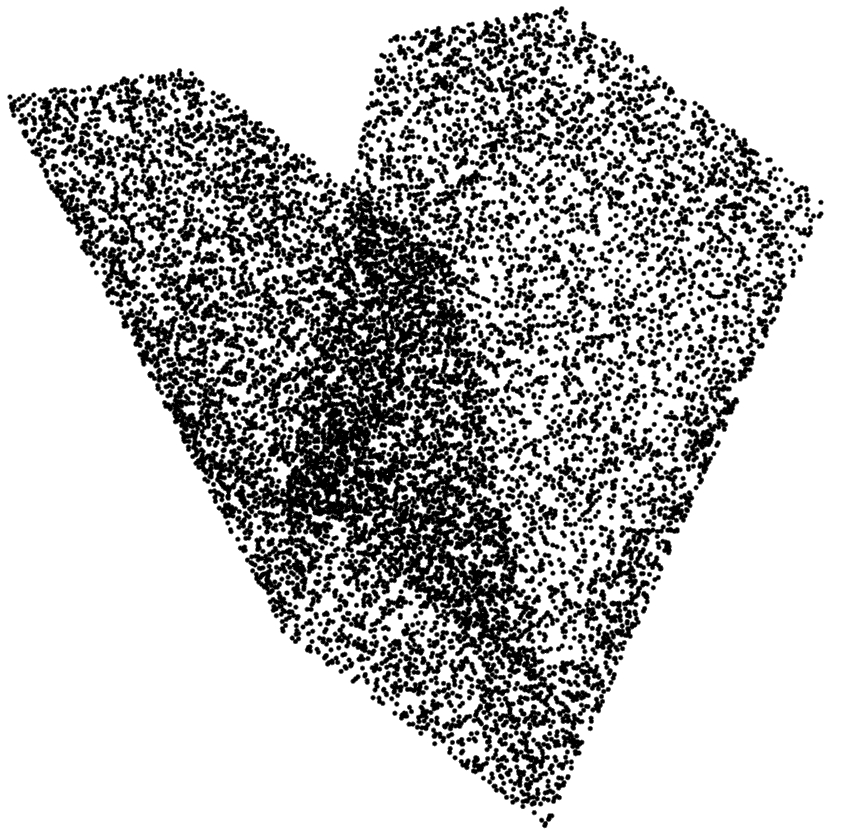}
    }
    
    \caption{(Three-dimensional example.) Point cloud data.} 
    \label{fig:Original point-cloud}
\end{figure}

\begin{figure}[t!]
    \centering
    \subfigure[ITM]{
    \centering
    \includegraphics[width=0.25\linewidth]{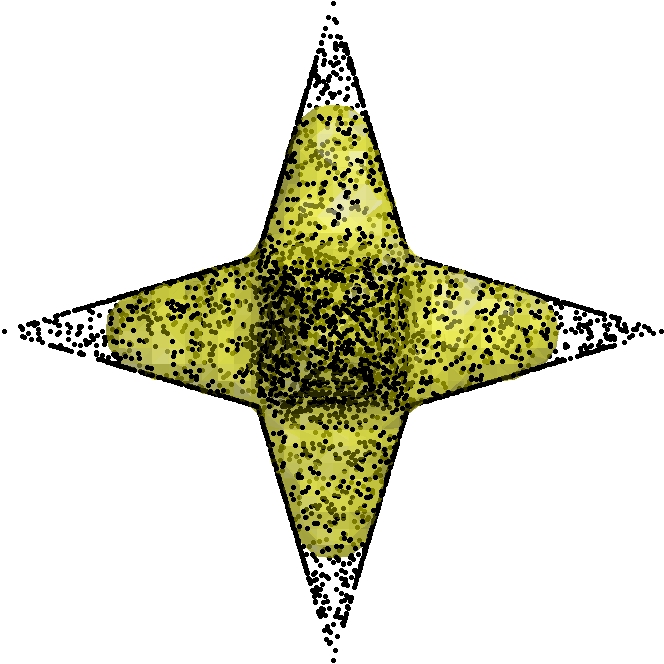}
    }
    \subfigure[Our method ]{
    \centering
    \includegraphics[width=0.25\linewidth]{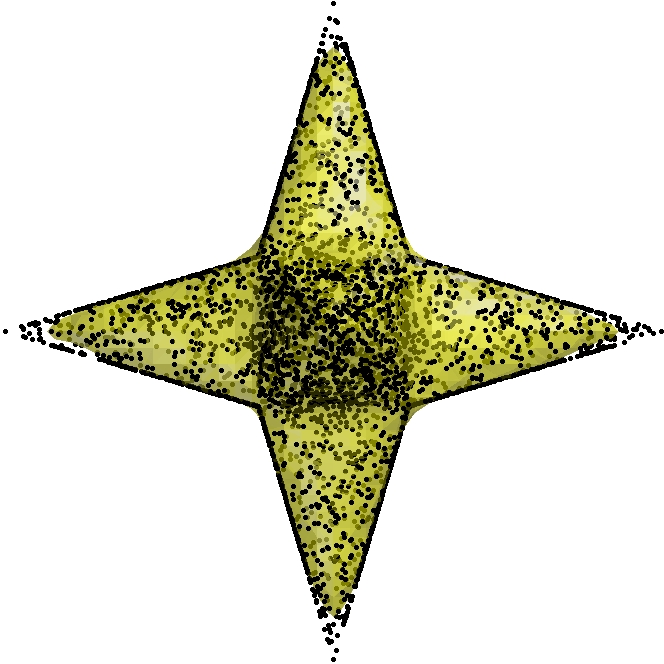}
    }
    \centering
    \subfigure[OSM+Curvature ]{
    \centering
    \includegraphics[width=0.25\linewidth]{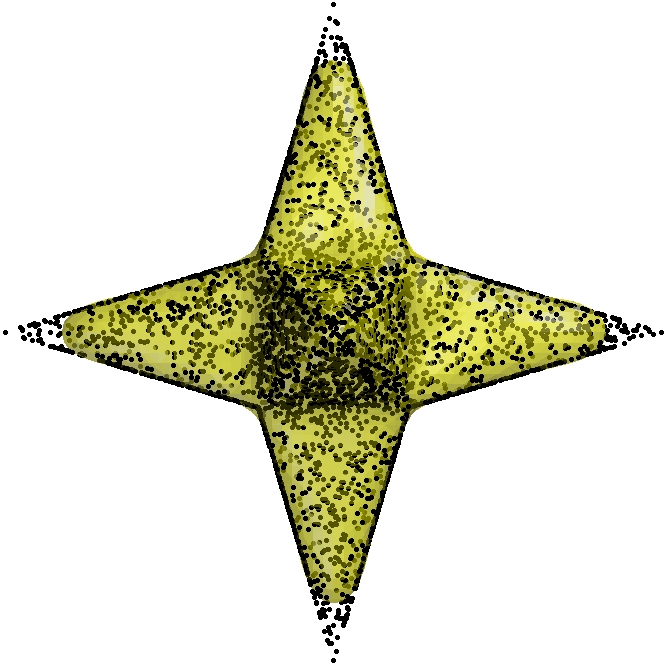}
    }
    \\
    \subfigure[Sectional view of picture(a), (b), (c) along the Y-Z plane]{
    \includegraphics[trim = 40 40 40 25,width=0.55\linewidth]{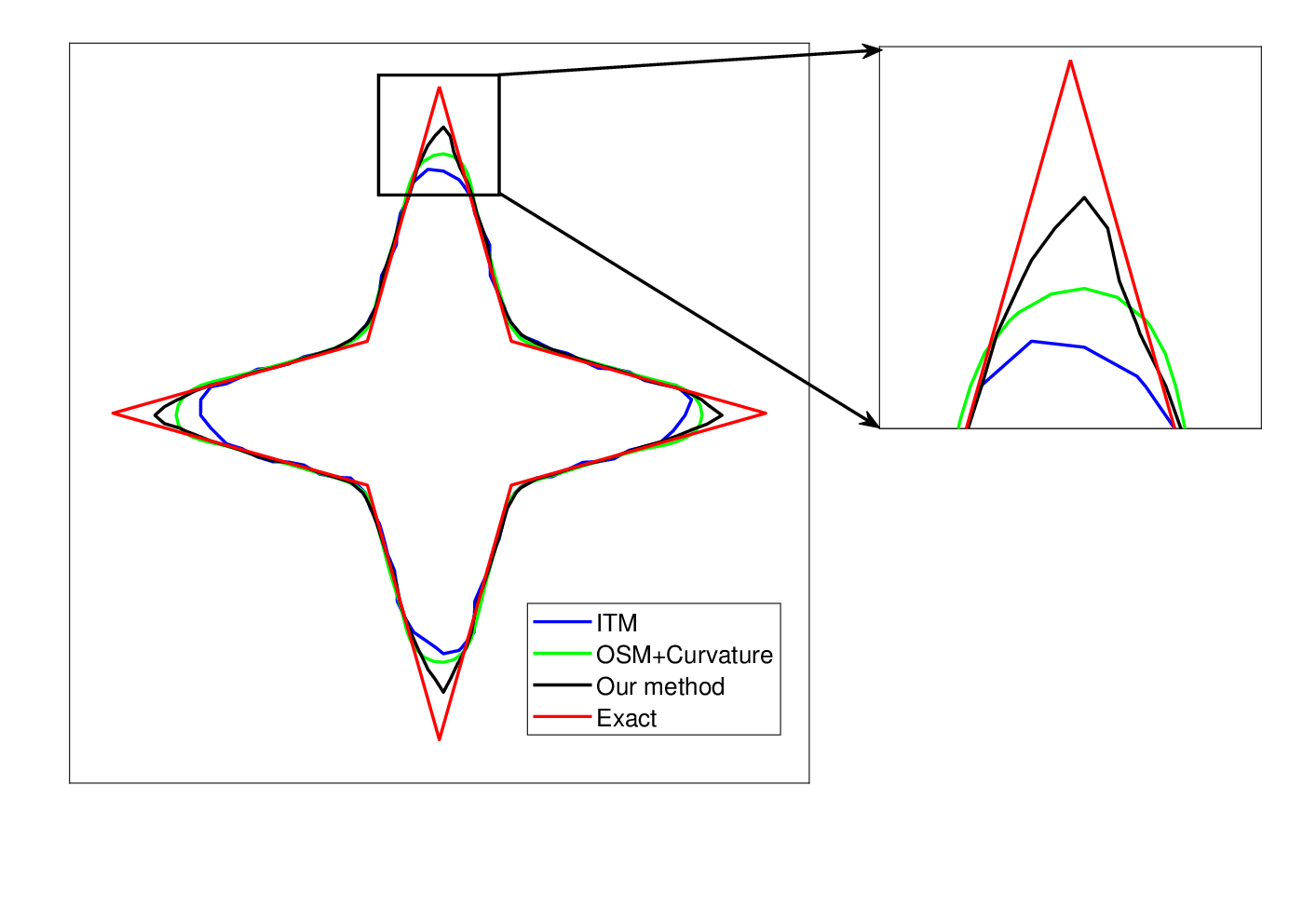}
    }
    \caption{(Three-dimensional example.) Reconstruction of the icositetrahedron (Figure \ref{fig:Original point-cloud}(a)) by ITM, OSM+Curvature and our method. }
    \label{fig:3DExperiment-1}
\end{figure}

For the icositetrahedron in Figure \ref{fig:Original point-cloud}(a), the surfaces reconstructed by the different methods are shown in Figure \ref{fig:3DExperiment-1}(a)-(c). In this experiment, we set $\tau=0.004$ for ITM and \(\eta = 5\) for OSM+Curvature. We fixed the number of iterations for OSM+Curvature at 200. Figure \ref{fig:3DExperiment-1}(d) presents the comparison of cross sections along the $Y-Z$ plane. In this experiment, both OSM+Curvature and our method effectively recover sharp-angle features, while our method yields slightly better results. This experiment shows that our method has a distinct advantage in reconstructing sharp convex corners.

\begin{figure}[t!]
    \centering
    \subfigure[ITM]{
    \includegraphics[width=0.25\linewidth]{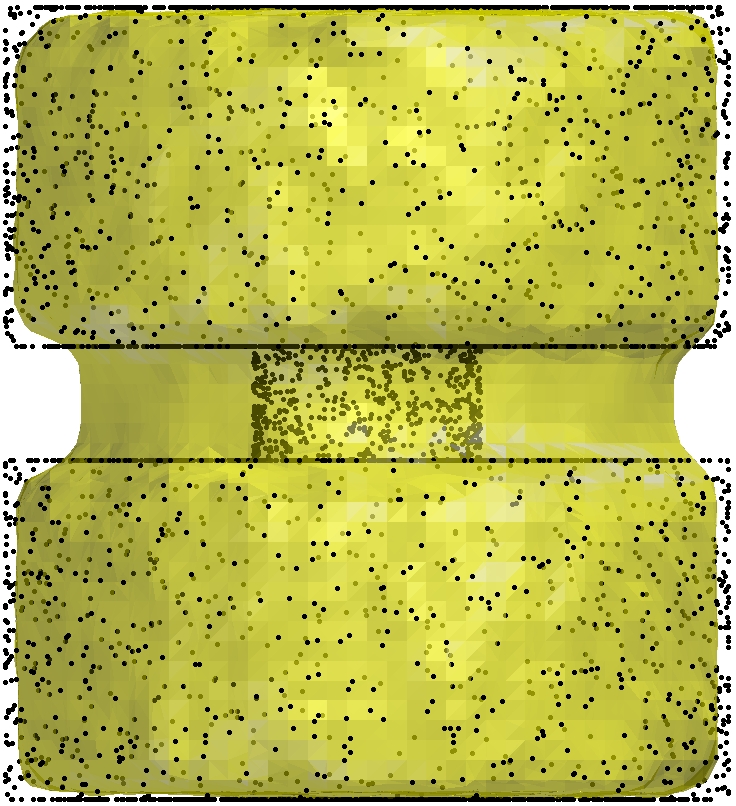}
    }
    \centering
    \subfigure[Our method]{
    \includegraphics[width=0.25\linewidth]{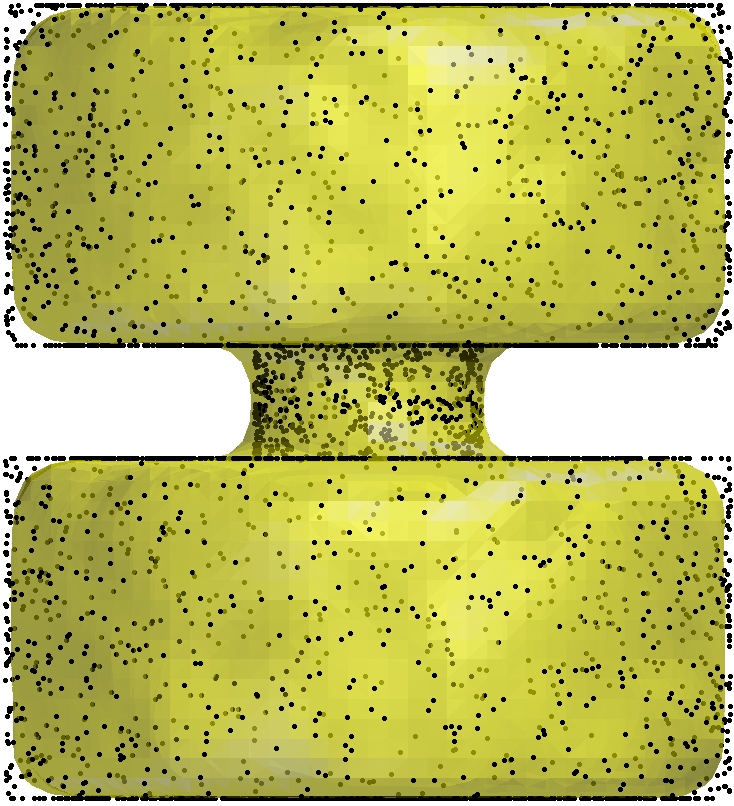}
    }
    \centering
    \subfigure[OSM+Curvature]{
    \centering
    \includegraphics[width=0.25\linewidth]{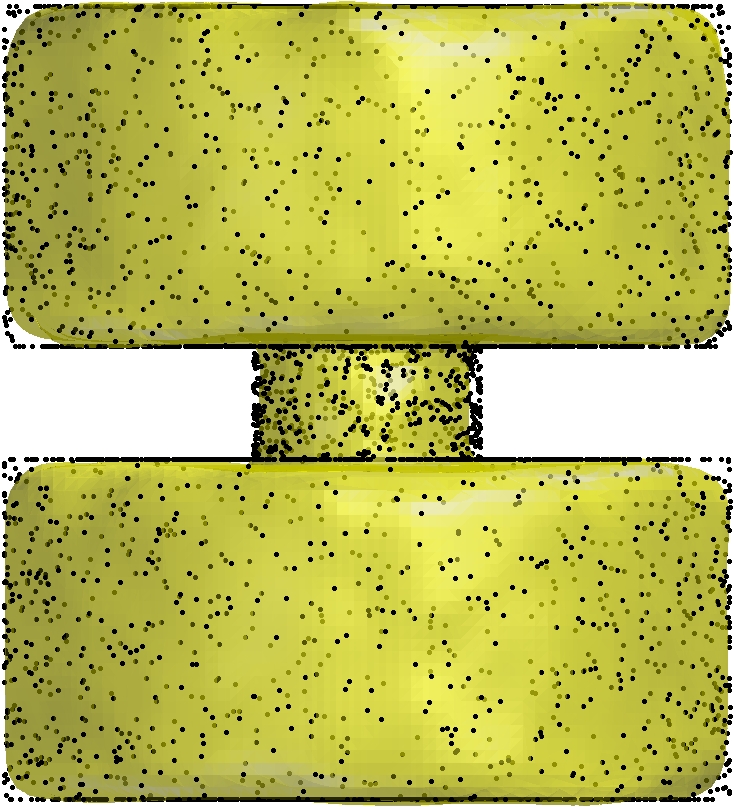}
    }
    \\
    \subfigure[Sectional view of picture(a), (b), (c) along the Y-Z plane]{
    \includegraphics[trim = 40 40 40 25,width=0.55\linewidth]{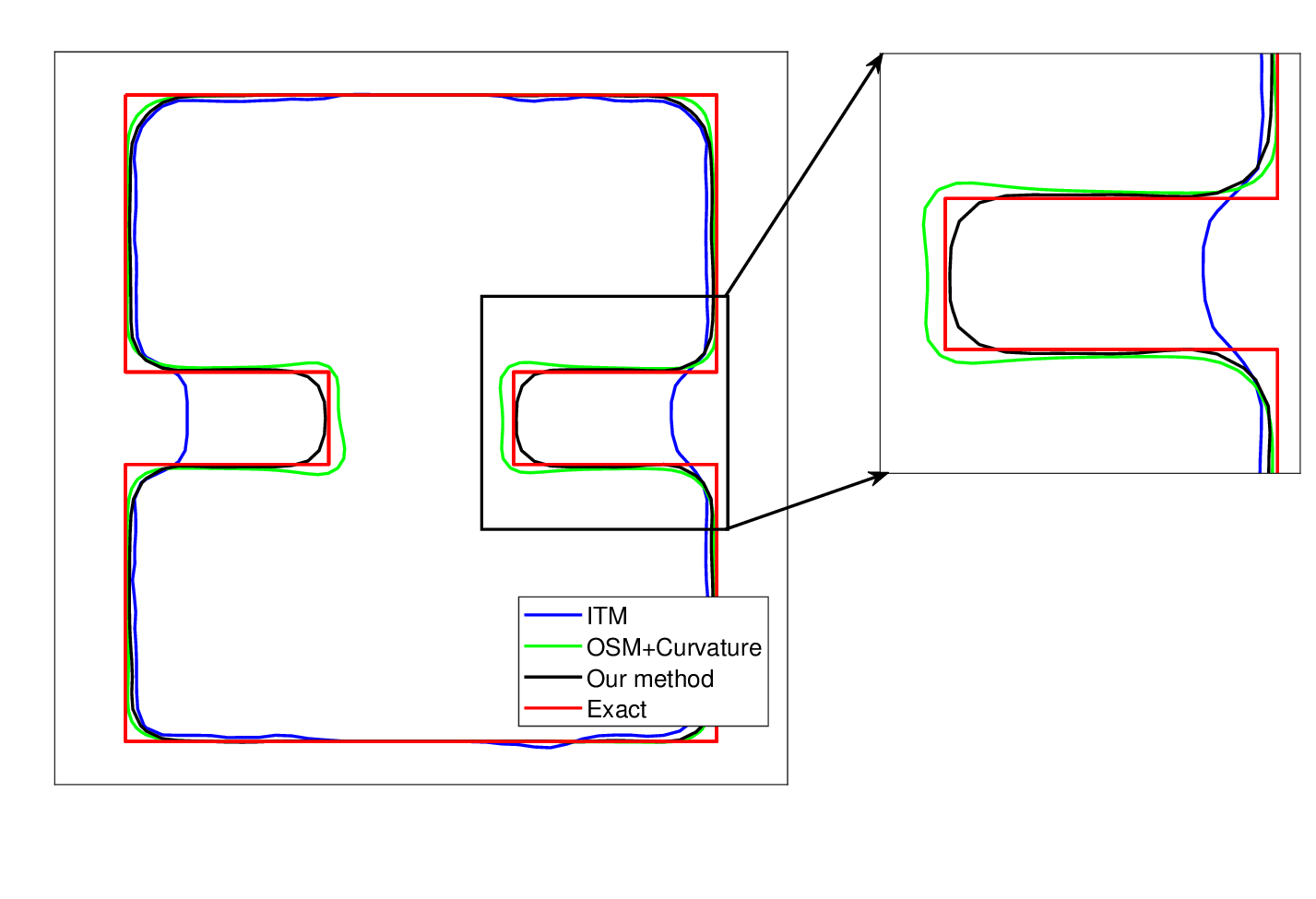}
    }
    \caption{(Three-dimensional example.) Reconstruction of the dumbbell--shaped surface (Figure \ref{fig:Original point-cloud}(b)) by ITM, OSM+Curvature and our method.}
    \label{fig:3DExperiment-2}
\end{figure}

Figure \ref{fig:3DExperiment-2} shows the surface reconstruction results for the dumbbell-shaped point cloud (Figure \ref{fig:Original point-cloud}(b)) by the three methods. We set $\tau=0.05$ for ITM and \(\eta = 5\) for OSM+Curvature. The number of iterations for OSM+Curvature was fixed at 1800. In this experiment, the ITM method fails to accurately approximate the concave connecting region of the dumbbell, whereas OSM+Curvature and our method achieve successful reconstruction. As demonstrated by the sectional views, both OSM+Curvature and our method exhibit superior performance in reconstructing concave surface regions.

\begin{figure}[t!]
    \centering
    \subfigure[ITM ]{
    \includegraphics[width=0.25\linewidth]{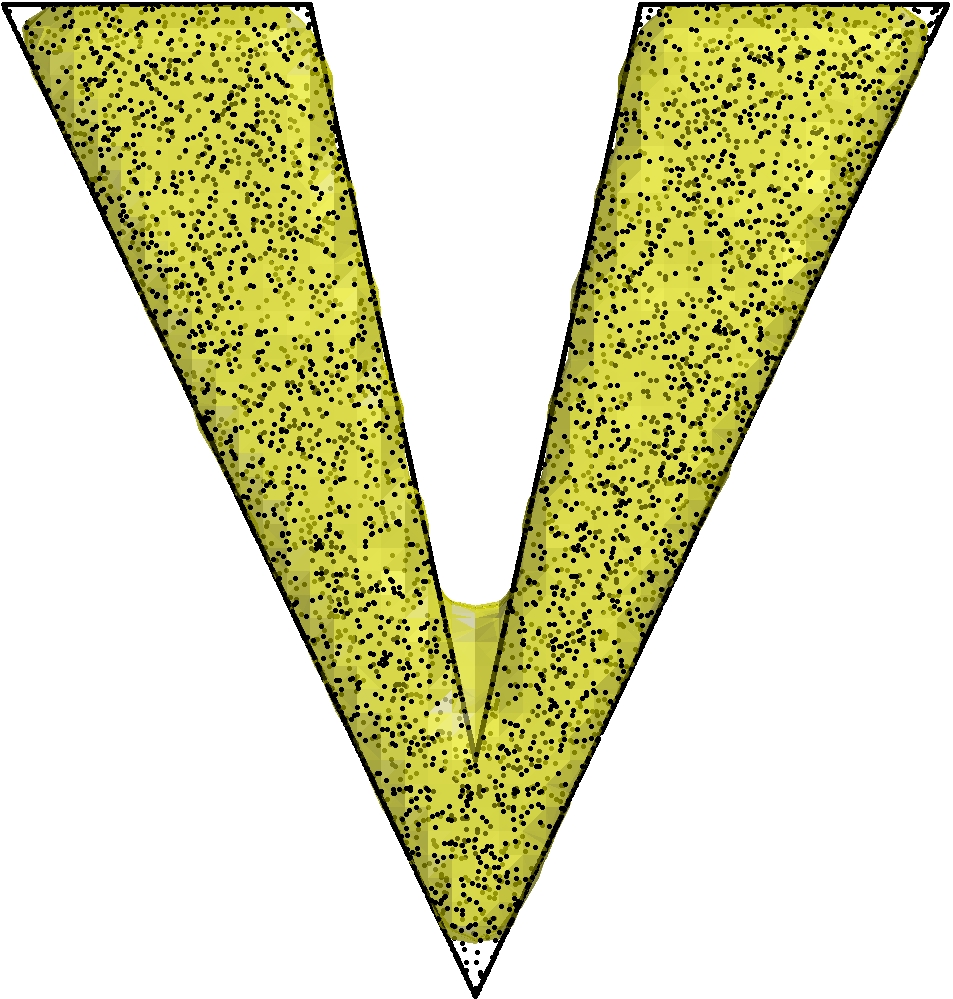}
    }
    \centering
    \subfigure[Our method ]{
    \includegraphics[width=0.25\linewidth]{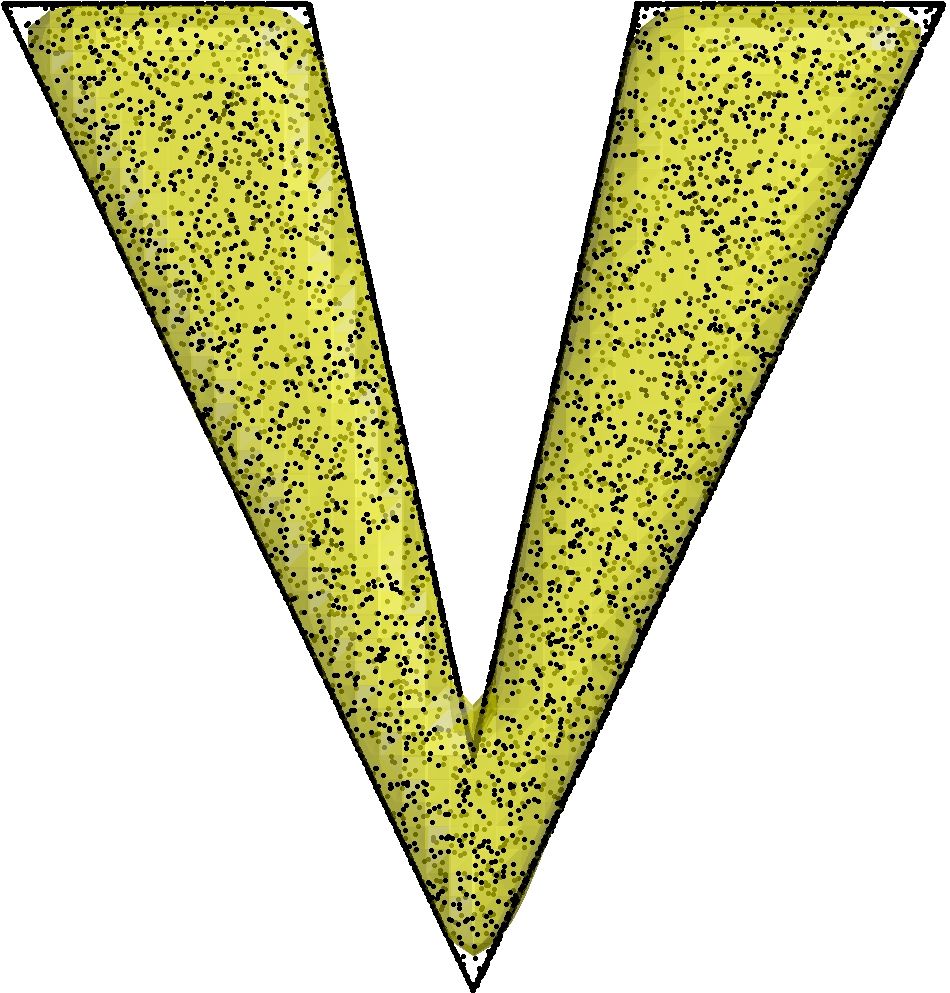}
    }
   \centering
    \subfigure[OSM+Curvature ]{
    \includegraphics[width=0.25\linewidth]{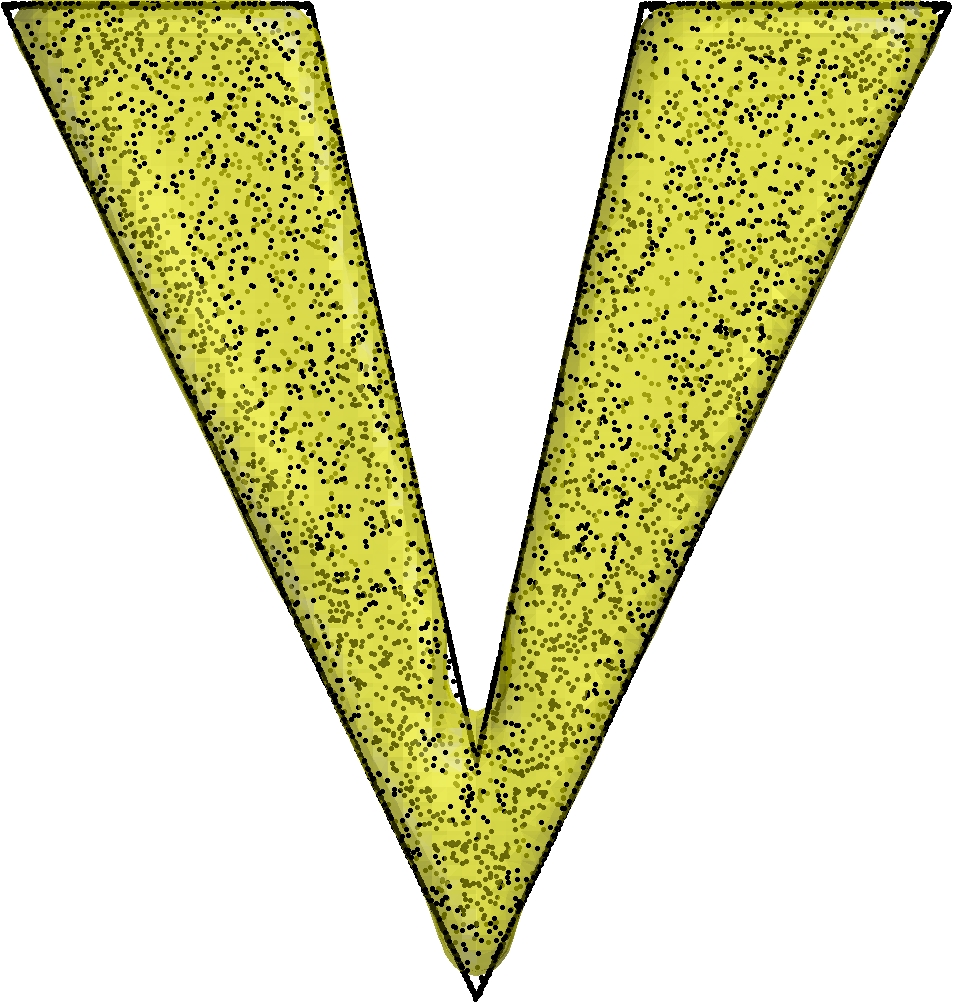}
    }
    \\
\subfigure[Sectional view of (a), (b), (c) along the Y-Z plane]{
    \includegraphics[trim = 40 40 40 25,width=0.55\linewidth]{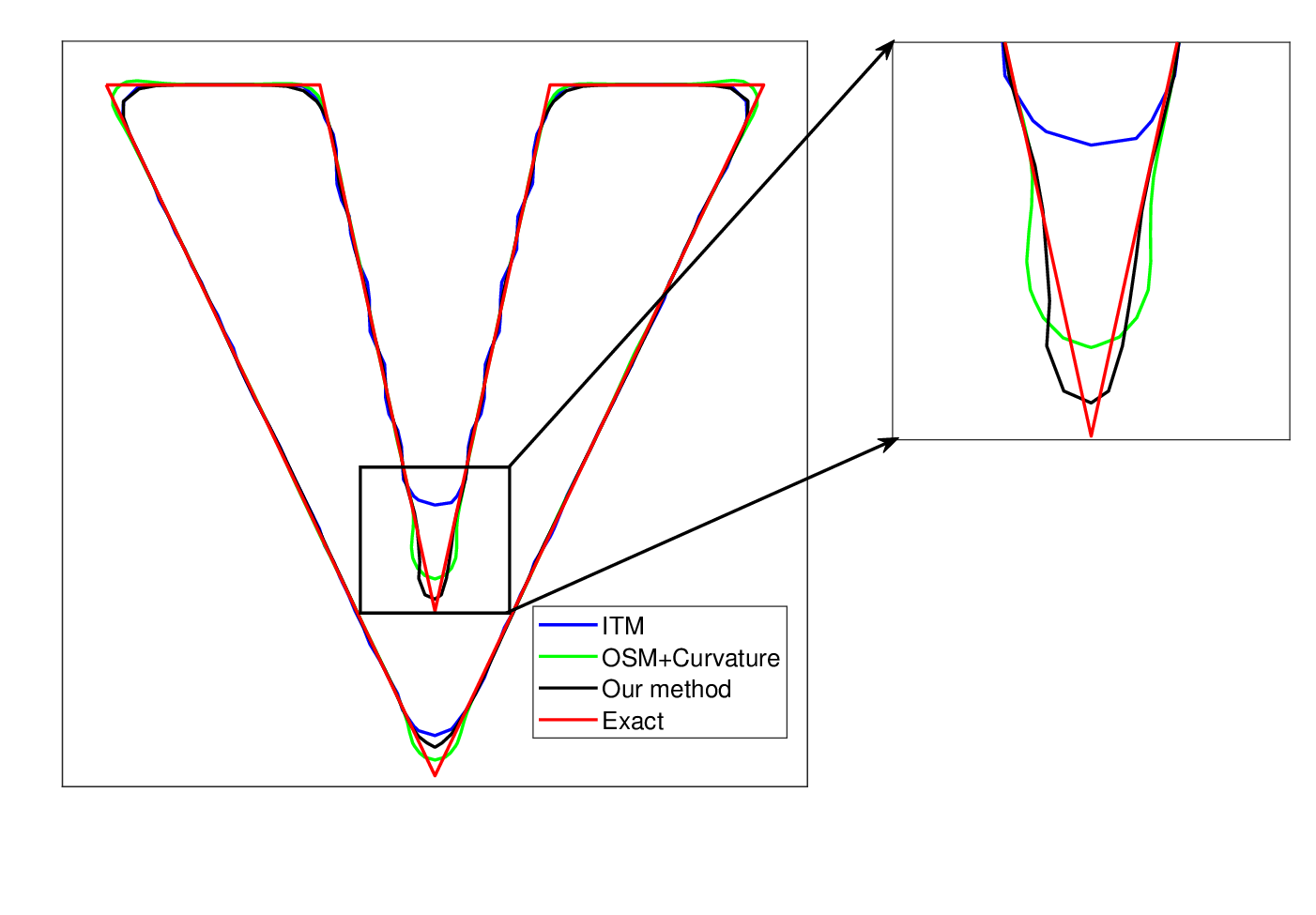}
    }
    \caption{(Three-dimensional example.) Reconstruction of the V-shape surface (Figure \ref{fig:Original point-cloud}(c)) by ITM, OSM+Curvature and our methods.}
    \label{fig:3DExperiment-3}
\end{figure}

Figure \ref{fig:3DExperiment-3} presents the reconstructed surfaces of the V-shaped surface (Figure \ref{fig:Original point-cloud}(c)), which has a distinct sharp concave corner. We set $\tau=0.01$ for ITM and \(\eta = 5\) for OSM+Curvature. The number of iterations for OSM+Curvature was fixed at 400. Our method achieves performance comparable to that of OSM+Curvature at the concave corner, and it outperforms the ITM method in handling this feature while preserving surface smoothness.

To compare the efficiency, the CPU time for all three experiments are presented in Table \ref{table:time2}. We observed that the CPU time of the ITM and our method differ by only a few seconds. In contrast, the OSM+Curvature method requires much more CPU time than the other two methods.

This comparison highlights that the proposed method yields more accurate surface reconstructions than the baseline ITM approach. Although the OSM+Curvature method achieves comparable accuracy, it requires substantially more processing time. Moreover, our method demonstrates strong generalization by requiring no parameter adjustments across all three experiments. Overall, these comparisons confirm our method's distinct advantages in sharp feature preservation, computational efficiency, and parameter robustness.

\begin{table}[t!]

\begin{center}
    
\begin{tabular}{|c | c | c | c |} 
 \hline
   & Figure \ref{fig:3DExperiment-1} & Figure \ref{fig:3DExperiment-2} & Figure \ref{fig:3DExperiment-3} \\[0.5ex]
 \hline
 ITM  & 11.35s & 6.53s & 7.04s \\ 
 \hline
OSM+Curvature  & 495.78s &  5105.63s  & 1080.10s\\
 \hline
Our method   & 12.36s & 12.77s &  13.85s \\
 \hline
\end{tabular}
\end{center}
\caption{(Three-dimensional example.) Required CPU time  to compute results shown in Figure \ref{fig:3DExperiment-1}-\ref{fig:3DExperiment-3}. }
\label{table:time2}
\end{table}

\section{Conclusion}
\label{sec.conclusion}
In this paper, we propose a novel method for surface reconstruction from point clouds. The proposed approach is based on a curvature-regularized model and represents the reconstructed surface using an indicator function. By introducing an auxiliary variable, we reformulate the original optimization problem as finding the steady-state solution to an initial value problem. An operator-splitting method was then developed to decompose the resulting problem into two subproblems, which can be efficiently solved. Owing to the indicator-function representation, one of these subproblems can be solved efficiently by iterative thresholding, thereby avoiding the costly reinitialization procedure required in traditional level set methods. Numerical experiments demonstrate that the proposed method is computationally efficient and is capable of accurately reconstructing sharp corners and concave features.

\bibliographystyle{abbrv}
\bibliography{ref}
\end{document}